%% file: MPR2_gerad.tex
\documentclass[gdweb]{geradwp}

\PassOptionsToPackage{hyphens}{url}

 \usepackage[ruled]{algorithm}
 \usepackage{algpseudocode}
\usepackage[english]{babel}
\usepackage{hyperref}
\usepackage{booktabs}
\usepackage{amsmath,amssymb,amsthm}
\usepackage{tikz}
\usepackage{subfigure}

\newtheorem{theorem}{Theorem}%  meant for continuous numbers

\newtheorem{lemma}{Lemma}
\newtheorem{assumption}{Assumption}

\GDcoverpagewhitespace{6.8cm}
\graphicspath{{Figures/}} % graphicx pkg setup
\hypersetup{colorlinks,%
citecolor={blue}, % Change for "black" with natbib
urlcolor={blue},
linkcolor={blue},
breaklinks={true}
}

\makeatletter
\ifthenelse{\isundefined{\ALG@name}}{}%
{%
\renewcommand{\ALG@name}{\sffamily\footnotesize Algorithm}
}
\makeatother
\ifthenelse{\isundefined{\SetAlCapNameFnt}}{}%
{% 
\SetAlCapNameFnt{\footnotesize}
\SetAlCapFnt{\sffamily\footnotesize}
}

\GDtitle{ A Multi-Precision Quadratic Regularization Method for Unconstrained Optimization with Rounding Error Analysis}
\GDmonth{D\'ecembre}{December}
\GDyear{2023}
\GDnumber{18}
\GDauthorsShort{D. Monnet, D. Orban}
\GDauthorsCopyright{Monnet, Dominique}
\GDsupplementname{Internet Appendix}
\GDrevised{D\'ecembre}{December}{2023}

\input{notations.tex}

\begin{document} 

\GDcoverpage

\begin{GDtitlepage}

\begin{GDauthlist}
\GDauthitem{Dominique Monnet \ref{affil:gerad}\GDrefsep\ref{affil:poly}}
\GDauthitem{Dominique Orban \ref{affil:gerad}\GDrefsep\ref{affil:poly}}
\end{GDauthlist}

\begin{GDaffillist}
\GDaffilitem{affil:gerad}{GERAD, Montr\'eal (Qc), Canada, H3T 1J4}
\GDaffilitem{affil:poly}{Polytechnique Montr\'eal, Montr\'eal (Qc), Canada, H3T 1J4}
\end{GDaffillist}

\begin{GDemaillist}
\GDemailitem{dominique.monnet@polymtl.ca}
\GDemailitem{dominique.orban@polymtl.ca}
\end{GDemaillist}

\end{GDtitlepage}

%% %%%%%%%%%%%%%%%%%%%%%%%%%%%%%%%%%%%%%%%%%%%%%%%%%%%%%%%
%% %%%%%%%%% Résumés, mots-clés, remerciements %%%%%%%%%%%
%% %%%%%%% Abstract, keywords, acknowledgements %%%%%%%%%%
%% %%%%%%%%%%%%%%%%%%%%%%%%%%%%%%%%%%%%%%%%%%%%%%%%%%%%%%%

\GDabstracts

\begin{GDabstract}{Abstract}
We propose a multi-precision extension of the Quadratic Regularization (R2) algorithm that enables it to take advantage of low-precision computations, and by extension to decrease energy consumption during the solve. 
The lower the precision in which computations occur, the larger the errors induced in the objective value and gradient, as well as in all other computations that occur in the course of the iterations.
The Multi-Precision R2 (MPR2) algorithm monitors the accumulation of rounding errors with two aims: to provide guarantees on the result of all computations, and to permit evaluations of the objective and its gradient in the lowest precision possible while preserving convergence properties.  MPR2's numerical results show that single precision offers enough accuracy for more than half of objective evaluations and most of gradient evaluations during the algorithm's execution. However, MPR2 fails to converge on several problems of the test set for which double precision does not offer enough precision to ensure the convergence conditions before reaching a first order critical point. That is why we propose a practical version of MPR2 with relaxed conditions which converges for almost as many problems as R2 and potentially enables to save about 50 \% time and 60\% energy for objective evaluation and 50 \% time and 70\% energy for gradient evaluation.

\end{GDabstract}

%\begin{GDabstract}{R\'esum\'e}
%R\'edigez votre r\'esum\'e ici...

%\paragraph{Mots cl\'es\,: }

%\end{GDabstract} 

%\begin{GDacknowledgements}
%Here

%\end{GDacknowledgements}

%% %%%%%%%%%%%%%%%%%%%%%%%%%%%%%%%%%%%%%%%%%%%%%%%%%
%% %%%%%%%%%%%%%%%% Article %%%%%%%%%%%%%%%%%%%%%%%%
%% %%%%%%%%%%%%%%%%%%%%%%%%%%%%%%%%%%%%%%%%%%%%%%%%%

\GDarticlestart

\section{Introduction}\label{sec1}

Over the recent years, the rise of Deep Neural Networks (DNN) and the continuous increase of the size of the networks has motivated researches to reduce the training cost.
This cost includes, for instance, the training time and the energy dissipation.
A natural way to do that is to perform the computations in low precision~\cite{sun2020ultra,wang2018training}, since dividing by two the number of digits roughly divide by two the computation time and by four the energy consumption~\cite{galal2010energy}.
Low-precision hardware architectures have been developed to serve the purpose of DNN fast training at lower energy cost.
In this context, it becomes possible to use several precision levels, and to extend the DNN training methods to multi-precision.

DNN training, \emph{i.e.}, minimizing the loss function, is a non-linear optimization problem.
Using low-precision computations come at the price of evaluating the objective function the gradient, and possibly higher derivative orders, with errors.
Convergence of non-linear algorithms under such evaluation errors has received attention in the recent years.
For the Trust Region (TR) methods, the state-of-the-art reference~\cite{conn2000trust} already includes the convergence conditions to inexact objective and gradient evaluations with deterministic error bounds.
These results were recently extended to the multi-precision case in~\cite{gratton2020note}, which proves the convergence of TR if infinite precision evaluations can be performed.
\cite{sun2022trust} consider TR methods with deterministic error bounds, and proposes to relax the acceptance step condition by including the error bounds.
This enables to keep the TR radius large, but the convergence only guaranteed to the neighborhood of a critical point.
TR convergence is also analyzed with probabilistic error bounds.
In this context, the convergence is provided in the sens of a probability which increases with the number of iteration.
TR methods are proved to converge with a probability 1 with exact objective function evaluation and probabilistic fully-linear model (which includes inexact order 1 and 2 derivative evaluations) in~\cite{bandeira2014convergence,gratton2018complexity}.
More recent works study convergence of TR by further including probabilistic error bounds on the objective function~\cite{blanchet2019convergence,sun2022trust,cao2022first}.

% The Black Box optimization (BB) community also investigates multi-precision methods.
% In the BB framework, surrogate models approximate the objective function with different precision levels, and are cheaper to evaluate.
% BB algorithms have been proposed with deterministic error bounds of the surrogate models \red{ask Charles for ref}~\cite{}, as well as probabilistic estimates of the objective function~\cite{audet2021stochastic}.

% \red{paragraph here on line-search methods ? see~\cite{cao2022first} for references on stochastic line search methods}

In the recent years, regularization methods have emerged as an alternative to TR.
These methods are based on local Taylor expansion models, to which is added a regularization parameters that controls the step size~\cite{cartis2012adaptive,cartis2018worst}, and have the same complexity as TR methods.
This complexity is $O(\epsilon^{-2})$ with $\epsilon$ the tolerance on the first order condition.
Recent works extend regularization methods' proof of convergence to inexact evaluation of objective function and gradient with deterministic and probabilistic error bounds~\cite{bellavia2018deterministic}, and deterministic bounds in the context of least-squares problems~\cite{bellavia2018levenberg}.
In particular, \cite{bellavia2018deterministic} can be viewed as the multi-precision regularization counterpart of~\cite{gratton2020note}.
In~\cite{gratton2021algorithm}, a composite problem incorporating a non-smooth term in the objective function is considered with deterministic error bounds on objective and gradient evaluations and shows that the complexity becomes $O(|\log(\epsilon)|\epsilon^{-2})$ for this class of problems.

In the context of multiple precision computations, mentioned in the first paragraph, considering only objective function and gradient (and possibly further derivative orders) is not enough to ensure convergence.
Indeed, running an algorithm with finite precision computations, and in particular with (very) low-precision computations, implies that any computed value is inexact.
The purpose of this paper is to propose a comprehensive convergence analysis of a multi-precision algorithm based on first order regularization methods, MPR2, by taking into account all the possible sources of error due to finite precision computations.
To the best of the author's knowledge, such a study has not been proposed for non-linear methods.
More generally, the aim of this paper is to point out some pitfalls of using low-precision computations.

This paper is organized as follows. 
Section~\ref{sec:pb_statement} introduces the problem we are aiming to solve, how the inexact evaluations of the objective function and gradient are modelled and the R2 algorithm in this context of inexact evaluations.
Section~\ref{sec:finite precision} introduces Floating Points (FP) computations. It lists the sources of errors induced by finite precision computations occurring in R2 and provides bounds on these errors.
Section~\ref{sec:MPR2} presents MPR2, the multi-precision variant of R2 that takes all the errors listed in Section~\ref{sec:finite precision} into account.
Section~\ref{sec:num_experiments} shows the numerical results obtain with MPR2.
Section~\ref{sec:conclusion} recaps the results presented in this paper and devises on future research directions.
Finally, the proof of convergence of MPR2 is given in Appendix~\ref{sec:proof}.

\section{Problem Statement and Background}
\label{sec:pb_statement}

We consider the problem
\begin{equation}%
	\label{eq:nlp}
	\minu{x\in \R^n} \, f(x),
\end{equation}
with $f: \R^n \to \R$ continuously differentiable.
We assume that evaluations and computations are performed in finite precision, and that a set of FP systems is available.
For example, modern hardware may offer access to half, single and double precision units.
As a consequence, it is not possible to evaluate $f$ nor $\nabla f$ exactly, but only finite precision approximations, denoted $\ffp$ and $\gfp$.
For more generality, we assume that models $\ef$ and $\eg$ of the evaluation errors are provided as functions of $x$:
\begin{equation}
	\label{eq:fgerr}
	|f(x)-\ffp(x)| \leq \ef(x), \quad \|\nabla f(x)-\gfp(x)\|\leq \eg(x)\|\gfp(x)\|,
\end{equation}
for all \(x \in \R^n\).
Note that $\ef$ and $\eg$ can account for any kind of evaluation errors such as, for example, gradient approximation in a derivative-free context, and are not limited to finite precision computation errors. 
%with $\fpreck{}$ and $\gpreck{}$ the indices of the floating point format used for evaluation $\ffp$ and $\gfp$.

We make the following assumption.
\begin{assumption}%
	\label{as:lips} 
	% The gradient of $f$ is Lipschitz continuous, i.e.,
	There exists \(L > 0\) such that \(\|\nabla f(x) - \nabla f(y)\| \leq L\|x-y\|\) for all \(x\), \(y \in \R^n\).
\end{assumption} 

The quadratic regularization method \cite{aravkin2022proximal}, called here R2, % and stated in Algorithm~\ref{alg:qrfp},
is based on the first-order model
\begin{equation*}
	T(x,s) = f(x) + \nabla f(x)^Ts.
\end{equation*}
Under Assumption~\ref{as:lips}~\cite{birgin2017worst},
\begin{equation}
	\label{eq:ft_diff}
	|f(x+s) - T(x,s)| \leq \tfrac{1}{2}L\|s\|^2.
\end{equation}

Since $f$ and $\nabla f(x)$ cannot be computed exactly, neither can $T$.
We must rely instead on the approximate Taylor series %with finite precision evaluations of the objective function and the gradient,
\begin{equation*}
	\approxT(x,s) = \widehat{f}(x) + \gfp(x)^T s.
\end{equation*} 
The regularized approximate Taylor series,
\begin{equation*}
	m(x,s) = \approxT(x,s) + \tfrac{1}{2}\sigma\|s\|^2,
\end{equation*}
is used as local model in the algorithm, where $\sigma > 0$ is the regularization parameter that controls the step size.
At each iteration, the step is chosen as the minimizer of the local model:
\begin{equation*}
	s \in \underset{s\in\R^n}{\argminu{}} \,  m(x,s) = \left\{-\frac{1}{\sigma} \widehat{g}(x)\right\}.
\end{equation*}

Algorithm~\ref{alg:qr} is an adaptation of the standard R2 algorithm that takes objective and gradient errors into account.
It is proved to converge to a first order critical point in~\cite{bellavia2018deterministic}, under the assumption that all other computations in the algorithm are performed in exact arithmetic.
In the algorithm we use $\eg(x_k)\leq \kappa_\mu$ to control the error on gradient evaluation instead of $\eg(x_k)\leq 1/\sigma_k$, which is employed in~\cite{bellavia2018deterministic}.
Doing so avoids demanding excessive accuracy on the gradient when $\sigma_k$ is large.

The stopping condition in Step~\(1\) of Algorithm~\ref{alg:qr} ensures that $\|\nabla f(x_k)\| \leq \epsilon$ while taking the error on the gradient evaluation into account.
In Step~\(2\), if the convergence condition $\eg(x_k)\leq \wgaugbound$ is not satisfied, the evaluation precision on the gradient is increased and the gradient is recomputed with a smaller $\eg(x_k)$ in Step~\(1\).
In Step~\(3\), a suitable evaluation precision for the objective function must be chosen such that $\ef(x_k)$ and $\ef(c_k)$ are lower than $\eta_0\overline{\Delta T}_k$ to ensure convergence.
% The input parameters have to be chosen as indicated at Step 0 to ensure the convergence.

\begin{algorithm}[ht]
	\caption{Quadratic regularization algorithm with inexact evaluations}%
	\label{alg:qr}
	\begin{algorithmic}
		\State \textbf{Step 0: Initialization:} Choose an initial point $x_0 \in \R^n$, initial $\sigma_0 > 0$, minimal value $\sigma_{\min} > 0$ for $\sigma$, final gradient tolerance $\epsilon > 0$,   % formula for $\gamma_n$. 
		and constants
		\begin{equation*}
			0<\eta_1\leq \eta_2<1,\quad 0<\gamma_1<1<  \gamma_2 \leq \gamma_3, \quad \eta_0<\tfrac{1}{2}\eta_1, \quad \eta_0+\tfrac{1}{2} \wgaugbound \leq \tfrac{1}{2}(1-\eta_2).
		\end{equation*}
		Set $k=0$, compute $\widehat{f}_0 = \ffp(x_0)$. Select a precision level for gradient evaluation.
		\State \textbf{Step 1: Check for termination:} If $k = 0$ or $\rho_{k-1} \geq \eta_1$ (previous iteration successful), compute $g_k = \gfp(x_k)$.
		Terminate if \begin{equation*}
			\|g_k\| \leq \dfrac{\epsilon}{1+\eg(x_k)}.
		\end{equation*}
		\State \textbf{Step 2: Step calculation: } Compute $s_k = -g_k/\sigma_k$.\\
		If $ \eg(x_k) > \wgaugbound$, increase gradient evaluation precision and go to Step 1.\\
		Compute $c_k = x_k+s_k$.\\
		Compute predicted reduction $\overline{\Delta T}_k := \approxT(x_k, s_k) - \ffp(x_k) = g_k^Ts_k$.
		
		\State \textbf{Step 3: Evaluate the objective function: }
		Choose the objective function evaluation precision such that
		$\ef(x_k)\leq  \eta_0 \tsdifffp_k$, and $\ef(c_k)\leq  \eta_0 \tsdifffp_k$.\\
		Compute $f_k^+ = \ffp(c_k)$ and recompute $f_k = \ffp(x_k)$ if evaluation precision different that $\ffp(c_{k-1})$.
		\State \textbf{Step 4: Acceptance of the trial point: } Define the ratio
		\begin{equation*}
			\rho_k = \dfrac{f_k - f_{k}^+}{\overline{\Delta T}_k}
		\end{equation*}
		If $\rho_k \geq \eta_1$, set $x_{k+1} =c_{k}$, $f_{k+1} = f_{k}^+$.
		Otherwise, set $x_{k+1} = \xkfp{k}$, $f_{k+1} = f_{k}$.
		
		\State \textbf{Step 5: Regularization  parameter update: }
		\begin{equation*}
			\sigma_{k+1} \in
			\begin{cases}
				[\max(\sigma_{\min},\gamma_1\sigma_k),\sigma_k] & \text{if } \rho_k \geq \eta_2\\
				
				[\sigma_k,\gamma_2\sigma_k] & \text{if }\rho_k \in [\eta_1,\eta_2) \\
				
				[\gamma_2\sigma_k,\gamma_3\sigma_k] & \text{if } \rho_k < \eta_1.
			\end{cases}
		\end{equation*}
		Increment $k \leftarrow k+1$, and go to Step~\(1\).
	\end{algorithmic}
\end{algorithm}

\section{Finite Precision Preliminaries}%
\label{sec:finite precision}

Finite precision computations induce not only inexact evaluations of the objective function and the gradient, but also other errors that must be taken into account.
In this section, the sources of errors due to finite precision computations are analyzed.
Mechanisms are added to Algorithm~\ref{alg:qr} to ensure the convergence in spite of these errors.
Algorithm~\ref{alg:qrfp}, introduced later in Section~\ref{sec:MPR2}, is the adapted version of  Algorithm~\ref{alg:qr} robust to finite precision errors.

\subsection{IEEE 754 Norm and Dot Product Error}

Let \(\F\) denote a FP system in base \(2\) with mantissa of length \(p\).
The \emph{machine epsilon} is 
\(
\epsilon_M := 2^{1-p}
\).
% The machine epsilon corresponds to the distance between 1 and the next machine representable number.
If the rounding strategy is round-to-nearest, the \emph{unit round-off} is 
\(
u := 2^{-p} = \tfrac{1}{2} \epsilon_M
\).
For \(x \in \R\), \(\fl(x) \in \F\) is the representation of \(x\) in \(\F\).
Provided no overflow occurs when representing \(x\), \(\fl(x) = x (1 + \delta)\) where $|\delta| \leq u$.
The IEEE standard~\cite{ieee754} requires that arithmetic operations produce a result such that 
\begin{equation}
	\label{eq:rounding_error}
	\fl(x\text{ op } y) = (x \text{ op } y)(1+\delta),
	\quad |\delta|\leq u,
	\quad \text{op}\in\{+,-,\times,/\},
	\ x, \, y \in \F.
\end{equation}

We make the following assumption.
\begin{assumption}
	Finite precision computations comply with the IEEE 754 norm, underflow and overflow do not occur during algorithm execution and the rounding mode is round to nearest (RN).
\end{assumption}

Successions of \(n\) arithmetic operations incur multiple errors of the form \(1 + \delta_i\) each, with $|\delta_i| \leq u$.
The notation used to study rounding error propagation is that of~\cite{higham2002accuracy}:
\begin{equation}
	\label{eq:err_propag}
	\prod_{i}^{n} (1+\delta_i)^{\pm 1} = (1+\theta_n),\quad |\theta_n|\leq \gamma_n,
\end{equation}
where several expressions for $\gamma_n$, summarized in Table~\ref{tab:gamma_formula}, were given over time.
Bounds such as~\eqref{eq:err_propag} can be overly pessimistic because they account for the worst-case propagation of errors.
%\todo{Inconsistent notation: $\gamma$ used for $\sigma$ update and bound on rounding error}
\begin{table}[ht]
	\begin{center}
		\caption{Formulas for $\gamma_n$ from the literature.}\label{tab:gamma_formula}%
		\begin{tabular}{cccc}
			\toprule
			$\gamma_n$ & $\dfrac{nu}{1-nu}$ & $\dfrac{nu}{1-nu/2}$ & $nu$ \\
			\midrule
			ref & \cite{higham2002accuracy} & \cite{castaldo2007error} & \cite{jeannerod2013improved} \\
			\bottomrule
		\end{tabular}
	\end{center}
\end{table}

One particular use of $\gamma_n$ is to bound the error on the computation of the dot product of \(x\) and \(y \in \R^n\)~\cite{higham2002accuracy}:
\begin{equation*}
	|\fl(x^T y) - x^T y|\leq \gamma_n |x|^T |y|,
\end{equation*}
where $|x|$ is the vector whose components are the absolute values of those of \(x\).

We make the following further assumption.
\begin{assumption}%
	\label{ass:gamma_n_bound}
	The unit roundoff of the lowest precision FP format $\mepsmax$ is such that $\gamma_{n+2} = (n+2)\mepsmax < 1$.
\end{assumption}
Assumption~\ref{ass:gamma_n_bound} is necessary to ensure the convergence of MPR2, and constrains the size of the problems that we consider in a given a FP format.
For example, $u = \mathcal{O}(10^{-8})$ for single precision, which limits the size of the problem to $n = \mathcal{O}(10^8)$ with the least pessimistic bound $\gamma_n = nu$ of Table~\ref{tab:gamma_formula}.

\subsection{Notation and Conventions}%
\label{subsec:notation}

Throughout, we use the following notation.
\begin{enumerate}
	\item We denote $\Pi = \{1, \ldots, \pi_{\max}\}$ the set of indices of available FP formats.
	We assume that $\Pi$ is ordered by decreasing unit round-off of the corresponding FP formats.
	For example, $\Pi = \{1,2,3\}$ corresponds to $\{\text{Float16},\text{Float32},\text{Float64}\}$.
	\item \label{not:delta} $\delta$ denotes a rounding error. An index is added if there is a need to keep track of a particular rounding error or to clarify the calculations.
	However, because we mostly focus on error bound $|\delta|<u$, there is often no need to keep track of particular rounding error and we abusively drop the index.
	For example, in the analysis
	$\fl(c/\fl(a+b)) = \fl(\dfrac{c}{(a+b)(1+\delta)}) = \dfrac{c}{a+b}\, \dfrac{1+\delta}{1+\delta}$, $\delta$ is employed to denotes two different rounding error bounded as $|\delta|\leq u$.
	Therefore, $\dfrac{c}{a+b}\, \dfrac{1+\delta}{1+\delta} \neq c/(a+b)$ but we write, following \eqref{eq:err_propag}, $\dfrac{c}{a+b}\, \dfrac{1+\delta}{1+\delta} = (a+b)/c\, (1+\theta_2)$ with $|\theta_2|<2u$.
	There is no need here to know which $\delta$ corresponds to $+$ or $/$ operators to perform the rounding error analysis.
	\item FP numbers are denoted with a hat, $\widehat{x}\in \F$.
	\item The notation tilde is used to denote a quantity that includes a rounding error.
	For example $\fl(\widehat{x}+\widehat{y}) = (\widehat{x}+\widehat{y})(1+\delta) = \widetilde{x}+\widetilde{y}$ with $\widetilde{x} = \widehat{x}(1+\delta)$ and $\widetilde{y} = \widehat{y}(1+\delta)$.
	Note that in general, $\widetilde{x}\notin\F$.
	\item $\theta_n$ models any rounding error propagation as in~\eqref{eq:err_propag}.
	\item $\vartheta_n$ denotes a real number bounded as $|\vartheta_n|\leq \gamma_n$.
	It is important to note that $\vartheta_n$ is not the result of accumulated rounding errors and is not expressed as $\theta_n$.
	As a consequence, $(1+\theta_n)(1+\delta) = (1+\theta_{n+1})$ but $(1+\vartheta_{n})(1+\delta) \neq (1+\vartheta_{n+1})$.
	\item $\fpreck{k}$ denotes the index of the FP format used to evaluate $\ffp$ at $\ckfp{k}$, and $\fprecrecompk{k}$ the one used to evaluate $\ffp$ at $\xkfp{k}$.
	Similarly, $\gpreck{k}$ denotes the index of the FP format used to evaluate $\gfp$.
	\item $\xpreck{k}$ denotes the index of the FP format of the current incumbent $\xkfp{k}$, and $\cpreck{k}$ the index of the FP format of the candidate $\ckfp{k}$.
	\item We use the notation \(\fl(x)\) when the FP system in which we wish to represent \(x\) is obvious from the context.
	Wherever it may be ambiguous, we write \(\fl(x, \pi)\), where \(\pi \in \Pi\).
	\item $\mepsk{c}{k}$ denotes the unit round-off related to the FP format corresponding to $\cpreck{k}$, and $\mepsk{g}{k}$ the unit round-off related to the FP format corresponding to $\gpreck{k}$.
	\item $\mepsmax$ (respectively $\mepsmin$) denotes the largest (resp.\ the smallest) unit roundoff among the available FP formats, i.e., that associated with the lowest (resp.\ the highest) precision.
	
	\item $\gamma_{n}$, $\alpha_n$, and $\beta_{n}$ (introduced later) are quantities that depend on the unit roundoff and accounts for rounding error of dot product and norm computations with finite precision.
	Since several FP formats are dealt with, the unit roundoff might be added as an argument to avoid confusion.
	For example, we denote $\gamma_n(\mepsk{x}{k})$ the quantity $\gamma_{n}$ defined with the unit roundoff corresponding to the FP format of $\xkfp{k}$.
	\item \label{not:abusive_vect}Given a vector $\widehat{x}\in \F^n$, we denote abusively,
	\begin{equation*}
		\begin{array}{ll}
			\xkfp{}(1+\delta_x) = (\xkfp{1}(1+\delta_x),\dots,\xkfp{n}(1+\delta_x))^T,\quad
			\xkfp{}\delta_x = (\xkfp{1}\delta_x,\dots,\xkfp{n}\delta_x)^T,
		\end{array}   
	\end{equation*}
	and vice-versa, where the same $\delta_x$ is ``broadcast" to every components of $\xkfp{}$.
	This abusive notation is also extended to $\delta$, $\theta_n$ and $\vartheta_{n}$.
\end{enumerate}

\textbf{Multi-format operations:} When performing computations with FP numbers in different formats, we assume that the operands are implicitly cast into the highest precision format among theirs and then the operation is performed.
The resulting FP number is therefore represented in this last format and the rounding errors are bounded by its unit round-off.
For example, if $\xkfp{k}$ is a vector of Float32 elements and $\skfp{k}$ a vector of Float64 elements, the computation of $\ckfp{k} := \fl(\xkfp{k}+\skfp{k})$ first consists in casting each element of $\xkfp{k}$ to Float64 and then performing the addition, which results in $(\xkfp{k}+\skfp{k})(1+\delta)$, i.e., a vector of Float64 elements where $\delta$ is bounded by the unit round-off of Float64.  
The implicit casting into another FP format of the arguments is done only in the context of the operation, $\xkfp{k}$ still has Float32 format after computing $\ckfp{k}$ in the above example.

\textbf{Representation of $\sigkfp{k}$:} We assume that $\sigkfp{k} := \fl(\sigma_k)$ is representable exactly in any of the available FP formats.
This can be achieved by setting $\sigma_0$ to a power of \(2\) and making sure that $\sigkfp{k}$ is multiplied by a power of 2 (positive of negative) at Step~\(5\) of Algorithm~\ref{alg:qrfp}.
In this way, we can safely ignore the FP format of $\sigkfp{k}$, and we assume that the FP format of $\skfp{k}$ is that of $\gkfp{k}$ when computed at Step~\(2\), as indicated in Table~\ref{tab:op_prec}.
%\textcolor{red}{La référence au tableau ne fonctionne pas à cause de la minipage.}

\textbf{Forbidden evaluations:} Evaluating $\ffp$ at $\xkfp{k}$ is forbidden if $\xpreck{k}>\fprecrecompk{k}$, as that would imply casting $\xkfp{k}$ to a lower precision format.
A rounding error would hence occur, and $\ffp$ would be evaluated at $\xkfp{k}(1+\delta)$ instead of $\xkfp{k}$.
In order to keep the proof of convergence theoretically sound, such evaluation is forbidden.
If $\xpreck{k}<\fprecrecompk{k}$, casting induces no rounding error.
%\textcolor{red}{Il y a ici une hypothèse que les systèmes sont emboîtés.}
As for multi-format operations, casting is implicit and the FP format of the argument of $\ffp$ remains the same before and after the evaluation of $\ffp$.
The index of the FP format of the returned value by the evaluation of $\ffp$ is $\fprecrecompk{k}$.
All the above applies to $\gfp$ similarly.
Since the format index of $\skfp{k}$ is $\gpreck{k}$, it follows that $\gpreck{k}\geq \xpreck{k}$ for all $k$, and the elements of $\xkfp{k} + \skfp{k}$ are in a FP format with index $\gpreck{k}$, as indicated in Table~\ref{tab:op_prec}.

\textbf{Defined values:} We assume that some values in Algorithm~\ref{alg:qrfp} are \emph{defined}, that is, ``computed'' in exact arithmetic.
These values are $\rho_k$, $\phi_k$, $\gamma_n$, $\beta_n$, $\mu_k$, and $\umixk{k}$, which is why those quantities are denoted without a hat.
Computing those values involve few operations, and therefore few rounding errors.
In practice, they can be computed in high precision cheaply, which may be assimilated to exact arithmetic for the present purposes.

Table~\ref{tab:op_prec} summarizes what are the indices of the FP formats of the values returned by operations and evaluations of $\ffp$ and $\gfp$.
Table~\ref{tab:fg_eval_prec} is limited to three FP formats but can be extended to any number of formats.
Table~\ref{tab:precs} summarizes the indices of the FP formats of the computed values.

\begin{table}[ht]
	\begin{minipage}[t]{0.49\linewidth}%
		\caption{Indices of FP format of values returned by operations.
			$-$ denotes an operation not occurring in Algorithm~\ref{alg:qrfp}.
			\label{tab:op_prec}
		}
		\centering
		\begin{tabular}{ccc}
			$+,\, / $ & $\xkfp{k}$ & $\gkfp{k}$
			\\ \toprule
			$\skfp{k}$ & $\gpreck{k}$ & ---
			\\ \midrule
			$\sigkfp{k}$ & --- & $\gpreck{k}$
			\\ \bottomrule
		\end{tabular}
	\end{minipage}
	\hfill
	\begin{minipage}[t]{0.49\linewidth}%
		\caption{Indices of FP format of values returned by evaluation of $\ffp$ and $\gfp$.
			A blank denotes a forbidden evaluation.
			\label{tab:fg_eval_prec}
		}
		\centering
		\begin{tabular}{ccccc}
			& \multicolumn{4}{c}{$\fpreck{k}$/$\gpreck{k}$}
			\\ \toprule
			& \multicolumn{1}{c}{}& 1 & 2 & 3 \\ \cline{3-5}
			$\xpreck{k}$ & 1 & 1 & 2 & 3 \\
			/ & 2 & & 2 & 3 \\ 
			$\cpreck{k}$ & 3 & & & 3
		\end{tabular}
	\end{minipage}
\end{table}

\begin{table}[ht]
	\caption{Indices of FP format of the computed values in Algorithm~\ref{alg:qrfp}.}%
	\label{tab:precs}
	\begin{tabular}{ccccccccccc}
		$\xkfp{k}$ & $\skfp{k}$ & $\ckfp{k}$ & $\fkfp{k}$ & $\fkcfp{k}$ & $\gkfp{k}$ & $\tsdifffp_k$ & $\fl(\|\xkfp{k}\|)$ & $\fl(\|\skfp{k}\|)$ & $\widehat{\|\gkfp{k}\|}$ & $\widehat{\phi}_k$
		\\ \toprule
		$\xpreck{k}$ & $\gpreck{k}$ & $\cpreck{k}$ & $\fpreck{k-1}/\fprecrecompk{k}$ & $\fpreck{k}$ & $\gpreck{k}$ & $\gpreck{k}$ & $\xpreck{k}$ & $\gpreck{k}$ & $\gpreck{k}$ & $\gpreck{k}$
	\end{tabular}
\end{table}

\subsection{Inexact Step and Backward-Error Taylor Series}%
\label{subsec:bet-decrease}

When computing the step $\skfp{k}$ at iteration $k$ in finite precision arithmetic (Step~\(1\) of Algorithm~\ref{alg:qr}), rounding errors occur and $\skfp{k}$ is not in the direction $\gkfp{k} = \gfp(\xkfp{k})$ but instead in a direction $\gd_k$ that includes those rounding errors.
Indeed,
\begin{equation}%
	\label{eq:sk_exp}
	\skfp{k} =
	\fl(-\frac{1}{\sigkfp{k}} \gkfp{k}) =
	-\frac{1}{\sigkfp{k}} \gd_k,
\end{equation}
with
\begin{equation}%
	\label{eq:gk_tilde}
	\gd_{k} = \gkfp{k}\odot\left((1+\delta_1),\dots,(1+\delta_n)\right),
\end{equation}
where $\odot$ is the element-wise multiplication operator.
%The $\delta$s in \eqref{eq:gk_tilde} are distinct but there is no need to keep track of them so we dropped the indices as explained in item \ref{not:delta}. of Section~\ref{subsec:notation}, and we employ the abusive notations details in item \ref{not:abusive_vect}.
%\textcolor{red}{le vecteur devrait-il être $((1 + \delta_1), \dots, (1 + \delta_n))$ ?}

Note that following the remarks of Section~\ref{subsec:notation}, the FP format index of $\gkfp{k}$ is $\gpreck{k}$ since $\gkfp{k}$ results from the evaluation of $\gfp$, and $\gpreck{k}$ is also the FP format of $\skfp{k}$ since we assume that $\sigkfp{k}$ is representable in any FP format (see Table~\ref{tab:precs}).

\begin{lemma}%
	\label{lem:g_actual_diff}
	For all \(k\), \(\|\gkfp{k}-\gd_k\| \leq \mepsk{g}{k}\|\gkfp{k}\|\) and $\|\gd_k\| \geq \|\gkfp{k}\|(1-\mepsk{g}{k})$.
\end{lemma} 

\begin{proof}
	Denote $\gkfp{k} = (\gkfp{k}^1,\dots,\gkfp{k}^n)$ and $\gd_k = (\gd_k^1,\dots,\gd_k^n)$.
	We have from~\eqref{eq:gk_tilde}
	\[
	\|\gkfp{k}-\gd_k\|^2 = 
	\sum_{i=1}^n (\gkfp{k}^i-\gkfp{k}^i(1+\delta_i))^2 \leq
	\sum_{i=1}^n (\mepsk{g}{k})^2 (\gkfp{k}^i)^2 =
	(\mepsk{g}{k})^2\|\gkfp{k}\|^2,
	\]
	which proves the first inequality.
	Similarly, because \(-u\leq \delta \leq u\) \(\implies\) \((1-u)^2\leq (1+\delta)^2 \leq (1+u)^2\),~\eqref{eq:gk_tilde} implies
	\begin{equation*}
		\|\gd_k\|^2  =
		\sum_{i=1}^n (\gkfp{k}^i(1+\delta_i))^2 \geq
		(1-\mepsk{g}{k})^2\|\gkfp{k}\|^2. \tag*{\qed}
	\end{equation*} 
	\renewcommand{\qed}{}  % don't need a second qed symbol in this proof.
\end{proof}

We define the \emph{Backward Error Taylor series} (BET) as
\begin{equation*}
	\pseudoT(\xkfp{k},s) = \ffp(\xkfp{k}) + \gd^Ts.
\end{equation*}
The terminology ``backward error''~\cite{higham2002accuracy} refers to the fact that $\gd$ includes the backward error due to the finite precision computation of $\skfp{k}$.
The BET series, being defined with $\gd^T$ the direction of the computed step $\skfp{k}$, has to be considered in place of $\approxT$ to prove the convergence of MPR2.
%The BET series can be viewed as the approximate Taylor series $\approxT$ including the rounding error.
%Indeed, $\pseudoT(x_k,s)$ is expressed with $\gd_k$, which is the direction of the inexact step $\skfp{k}$.
% As such, considering BET series $\pseudoT$ instead of $\approxT$ makes it easier to prove the convergence of the algorithm when considering rounding errors due to finite precision computations (see Section~\ref{sec:proof}).

Because $\gd$ is unknown, so is $\pseudoT$.
However, as Lemma~\ref{lem:pseudo_reduction} states, computing the decrease $\tsdifffp_k$ in the approximate Taylor series $\approxT$ with finite precision arithmetic provides the decrease $\widetilde{\Delta T}_k$ in the unknown BET series with a relative error.

\begin{lemma}%
	\label{lem:pseudo_reduction}
	Let $\tsdifffp_k := \fl\left( \approxT(\xkfp{k},0) - \approxT(\xkfp{k},\skfp{k}) \right)$ and \(\widetilde{\Delta T}_k := \pseudoT(\xkfp{k}, 0) - \pseudoT(\xkfp{k}, \skfp{k})\).
	Then,
	\begin{equation*}
		\tsdifffp_k =
		% \left(\pseudoT(x_k,0) - \pseudoT(x_k,\skfp{k})\right) (1+\vartheta_{n+1}) =
		\widetilde{\Delta T}_k (1+\vartheta_{n+1}) =
		-\gd_k^T \skfp{k} (1+\vartheta_{n+1}) =
		\dfrac{\|\gd_k\|^2}{\sigkfp{k}}(1+\vartheta_{n+1})
	\end{equation*}
	with $|\vartheta_{n+1}| \leq \gamma_{n+1}$.
	Furthermore, $\tsdifffp_k\geq 0$.
\end{lemma}

\begin{proof}
	Using~\eqref{eq:sk_exp},~\eqref{eq:gk_tilde} and the fact that \(\sigkfp{k}\) is representable exactly, one has \cite{higham2002accuracy},
	\begin{equation*}
		\begin{aligned}
			\tsdifffp_k = \fl(-\gkfp{k}^T \skfp{k}) & = - \fl\left( \left( \gd_k^1/(1+\delta_1),\dots,\gd_k^n/(1+\delta_n)\right)^T\,(-\gd_k/\sigkfp{k}) \right) \\
			& = \dfrac{(\gd_k^1)^2}{\sigkfp{k}}\dfrac{(1+\delta)^n}{(1+\delta_1)} + \dfrac{(\gd_k^2)^2}{\sigkfp{k}}\dfrac{(1+\delta)^n}{(1+\delta_2)} + \cdots + \dfrac{(\gd_k^n)^2}{\sigkfp{k}}\dfrac{(1+\delta)^2}{(1+\delta_n)}\\
			& = \dfrac{(\gd_k^1)^2}{\sigkfp{k}}(1+\theta_{n+1}) + \dfrac{(\gd_k^2)^2}{\sigkfp{k}}(1+\theta_{n+1}) + \cdots + \dfrac{(\gd_k^n)^2}{\sigkfp{k}}(1+\theta_3).
		\end{aligned}
	\end{equation*}
	According to~\eqref{eq:err_propag}, the absolute value of each perturbation term $\theta$ in the sum above is bounded by $\gamma_{n+1}$.
	Since $\sigkfp{k}>0$ and we assume $\gamma_{n+1}<1$, we have \(1 - \gamma_{n+1} \leq 1 + \theta_j \leq 1 + \gamma_{n+1}\) for \(j = 3, \ldots, n+1\).
	It follows that, on the one hand,
	\begin{equation*}
		\begin{aligned}
			\tsdifffp_k & \geq \dfrac{(\gd_k^1)^2}{\sigkfp{k}}(1-\gamma_{n+1}(\mepsk{g}{k})) + \dfrac{(\gd_k^2)^2}{\sigkfp{k}}(1-\gamma_{n+1}(\mepsk{g}{k})) + \cdots + \dfrac{(\gd_k^n)^2}{\sigkfp{k}}(1-\gamma_{n+1}(\mepsk{g}{k}))\\
			& \geq \frac{1}{\sigkfp{k}} \gd_k^T \gd_k (1-\gamma_{n+1}(\mepsk{g}{k}))\\
			& \geq \gd_k^T\skfp{k}(1-\gamma_{n+1}(\mepsk{g}{k})),
		\end{aligned}
	\end{equation*}
	and on the other hand,
	\begin{equation*}
		\tsdifffp_k \leq \gd_k^T\skfp{k}(1+\gamma_{n+1}(\mepsk{g}{k})).
	\end{equation*}
	Therefore, one has 
	\begin{equation*}
		\gd_k^T\skfp{k}(1-\gamma_{n+1}(\mepsk{g}{k})) \leq \tsdifffp_k \leq \gd_k^T\skfp{k}(1+\gamma_{n+1}(\mepsk{g}{k})),
	\end{equation*}
	which proves that there exists $\vartheta_{n+1}$ such that $|\vartheta_{n+1}| \leq \gamma_{n+1}(\mepsk{g}{k})$ satisfying
	\begin{equation*}
		\tsdifffp_k  = -\gd_k^T\skfp{k}(1+\vartheta_{n+1}).
	\end{equation*}
	Since $\skfp{k} = - \gd_k/\sigkfp{k}$, one has
	$-\gd_k^T\skfp{k} = \|\gd_k\|^2/\sigkfp{k}$ and,
	\begin{equation*}
		\tsdifffp_k  = -\gd_k^T\skfp{k}(1+\vartheta_{n+1}) = \dfrac{\|\gd_k\|^2}{\sigkfp{k}}(1+\vartheta_{n+1}),
	\end{equation*}
	which establishes the first statement of the lemma.
	
	Assumption~\ref{ass:gamma_n_bound} ensures that
	$1+\vartheta_{n+1}\geq 1-\gamma_{n+1}(\mepsk{g}{k}) \geq 0$.
	As a consequence, 
	\begin{equation}
		\label{eq:ts_pos}
		\tsdifffp_k  = \dfrac{\|\gd_k\|^2}{\sigkfp{k}}(1+\vartheta_{n+1}) \geq \dfrac{\|\gd_k\|^2}{\sigkfp{k}}(1-\gamma_{n+1}(\mepsk{g}{k})) \geq 0. \tag*{\qed}
	\end{equation}
	\renewcommand{\qed}{}
\end{proof} 

Although $\vartheta_{n+1}$ is bounded by $\gamma_{n+1}$, it is not, strictly speaking, the result of rounding error propagation since it is not expressed as in~\eqref{eq:err_propag}.
Rather, it represents the relative error between the decrease of the BET series $\widetilde{\Delta T}_k$ and the computed value of the truncated Taylor series decrease $\tsdifffp_k$.

%\textcolor{red}{Ce paragraphe n'est pas très clair.}
%To provide a better understanding of the BET series, consider the ratio $\rho_k$.
%Ideally, there is no evaluation error and the local order one model of the objective function would be $T$.
%However, it is necessary to take into account the gradient evaluation error, leading to define $\approxT$ as the not ideal model, but the one available and from which is computed the step $\skfp{k}$.
%If we assume that no rounding error occurs, then $\tsdifffp_k = \Delta\approxT_k$ and the ratio $\rho_k$ tells how well $\approxT$ approximates locally $f$.
%Taking into account the inexact computations, $\approxT$ is not the actual model to consider since $\skfp{k}$ is not in the direction of $g_k$, but in the one of $\gd_k$.
%As a consequence, $\pseudoT$ is the actual local model to consider.
%Defining $\rho_k$ as the decrease of the objective divided by $\tsdifffp_k = \left(\pseudoT(x_k,0) - \pseudoT(x_k,\skfp{k})\right)(1+\vartheta_{n+1})$, $\rho_k$ indicates how well the BET series approximates locally $f$, but with a ``measurement error'' on the model decrease being $(1+\vartheta_{n+1})$.

In order to streamline notation, we denote $\alpha_n$ an upper bound on $1 / (1+\vartheta_{n+1})$ defined as,
\begin{equation}%
	\label{eq:alphan}
	\alpha_{n+1} := \frac{1}{1-\gamma_{n+1}}\geq \frac{1}{1+\vartheta_{n+1}}>0.
\end{equation}
Note that Assumption~\ref{ass:gamma_n_bound} ensures that $\alpha_n$ is well-defined and that $\dfrac{1}{1+\vartheta_{n+1}}>0$.

\subsection{Actual Candidate and Search Direction}%
\label{subsec:candidate}

With finite precision computations, the computed candidate is
\begin{equation}
	\label{eq:actual_candidate}
	\ckfp{k} =
	\fl(\xkfp{k}+\skfp{k}) =
	(\xkfp{k}+\skfp{k})(1+\deltak{+}{k}) = \xkfp{k}+(\deltak{+}{k}\xkfp{k}+(1+\deltak{+}{k})\skfp{k}),
\end{equation}
with $|\deltak{+}{k}|\leq \mepsk{g}{k}$ (see Section~\ref{subsec:notation} for details).
As a consequence, the computed step $\skfp{k}$ is not the actual step.
The actual step is $  \xkfp{k}\deltak{+}{k} + \skfp{k}(1+\deltak{+}{k})$, and therefore the actual search direction is not $-\gd_k$ but $\sigkfp{k}(\ckfp{k} - \xkfp{k}) = -\gd_k(1+\deltak{+}{k}) +\sigkfp{k} \xkfp{k}\deltak{+}{k}$.
Figure~\ref{fig:actual_candidate} illustrates how rounding errors due to step and candidate computation affect the computation of the candidate $\ckfp{k}$.

\begin{figure}[ht]%
	\centering
	\input{figures/actual_candidate}
	\caption{Difference between the exact candidate $c_k$ and the computed candidate $\ckfp{k}$.%
		\label{fig:actual_candidate}
	}  % (without casting into $\cpreck{k}$).
\end{figure}
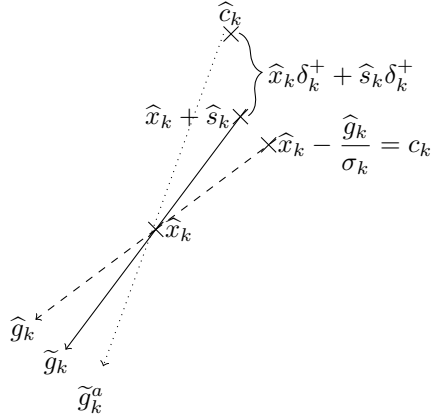

In~\eqref{eq:actual_candidate}, following the remarks of Section~\ref{subsec:notation}, the index of the FP format of $\ckfp{k}$ is $\gpreck{k}$.
We implement a mechanism that casts $\ckfp{k}$ into a FP format of index $\cpreck{k}$, which can be chosen freely.
This enables to lower the precision of the FP format of $\ckfp{k}$, and therefore those for $\ffp$ and $\gfp$.
Indeed, $\ffp$ and $\gfp$ cannot be evaluated at $\xkfp{k}$ or $\ckfp{k}$ in a FP format index lower than $\xpreck{k}$ or $\cpreck{k}$, respectively (see forbidden evaluations in Section~\ref{subsec:notation}), and $\xkfp{k+1} = \ckfp{k}$ if iteration $k$ is successful.
This mechanism is necessary to allow the precision for the evaluation of $\ffp$ and $\gfp$ to decrease over the iterations.
The extra cast of $\ckfp{k}$ is not represented in Figure~\ref{fig:actual_candidate} for simplicity.

Casting $\ckfp{k}$ into $\cpreck{k} < \gpreck{k}$ induces the additional rounding error
\begin{equation}
	\label{eq:casted_c}
	\fl(\ckfp{k},\cpreck{k}) = \ckfp{k}(1+\deltak{c}{k}) = (\xkfp{k}+\skfp{k})(1+\deltak{+}{k})(1+\deltak{c}{k}) =  (\xkfp{k}+\skfp{k})(1+\dmixk{k}),
\end{equation}
with $|\deltak{c}{k}|\leq \mepsk{c}{k}$ and,
\begin{equation*}
	\dmixk{k} = \deltak{+}{k} + \deltak{c}{k} +\deltak{+}{k}\deltak{c}{k}.
\end{equation*}
In the following, we consider $\dmixk{k}$ as a single virtual rounding error that is bounded by,
\begin{equation}%
	\label{eq:ukprime_def}
	\umixk{k} = \mepsk{g}{k} + \mepsk{c}{k} + \mepsk{g}{k}\mepsk{c}{k}.
\end{equation}
Note that if $\cpreck{k}\geq \gpreck{k}$, no rounding error due to the casting of $\ckfp{k}$ occurs and one simply has $\dmixk{k} = \deltak{+}{k}$ and $\umixk{k} = \mepsk{g}{k}$.

From~\eqref{eq:casted_c}, we define the \emph{actual step} as,
\begin{equation}%
	\label{eq:sact_def}
	\skact{k} := \skfp{k}(1+\dmixk{k}) + \xkfp{k}\dmixk{k},
\end{equation}
which satisfies $\fl(\ckfp{k},\cpreck{k}) = \xkfp{k}+\skact{k}$.
The \emph{actual search direction} is
\begin{equation}
	\label{eq:gact_def}
	-\gkact{k} := \sigkfp{k}\skact{k}.
\end{equation}

Since we have no relationship between $\gkfp{k}$ and $\xkfp{k}$, we do not have a relationship between $\skact{k}$ and $\gkfp{k}$ as we did in the previous section with $\gkfp{k}$ and $\gd_k$.
% This is an issue to prove the convergence of Algorithm~\ref{alg:qrfp}.
In order to deal with the inexact candidate computation, we introduce $\phi_k$ as,
\begin{equation}
	\label{eq:phi_k_def}
	\dfrac{\|\xkfp{k}\|}{\|\skfp{k}\|} \leq \phi_k.
\end{equation}  
Lemma~\ref{lem:actual_step_norm} provides an upper bound on $\|\skact{k}\|$ in terms of $\|\skfp{k}\|$ by taking advantage of $\phi_k$, enabling us to get rid of $\xkfp{k}$ in the expression that will prove useful later for the convergence analysis of MPR2.

\begin{lemma}%
	\label{lem:actual_step_norm}
	For all $k>0$,
	\begin{equation*}
		\|\skact{k}\| \leq \|\skfp{k}\|(1+\lambda_k),
	\end{equation*}
	with 
	\begin{equation}
		\label{eq:lambda_def}
		\lambda_k := \umixk{k}(\phi_k+1),
	\end{equation}
	where \(\umixk{k}\) is defined in~\eqref{eq:ukprime_def}.
\end{lemma}

\begin{proof}
	We combine~\eqref{eq:sact_def}, the triangle inequality and~\eqref{eq:phi_k_def}, and obtain
	\begin{align}
		\|\skact{k}\| & = \|\skfp{k}(1+\dmixk{k}) + \xkfp{k}\dmixk{k}\| \nonumber \\
		&\leq \|\skfp{k}\|(1+\dmixk{k})+\|\xkfp{k}\|\dmixk{k} \nonumber \\
		& \leq \|\skfp{k}\|(1+\umixk{k})+\|\xkfp{k}\|\umixk{k} \nonumber \\
		& \leq \|\skfp{k}\|(1+\umixk{k})+\phi_k\|\skfp{k}\|\umixk{k} \nonumber \\
		& = \|\skfp{k}\|(1+\lambda_{k}). \tag*{\qed}
	\end{align}
	\renewcommand{\qed}{}
\end{proof}

If $\phi_k$ is large, $ \xkfp{k} \dmixk{k}$ is not negligible compared to $\skfp{k}$ and the difference between $-\gkfp{k}$ and the actual search direction can be large (see Figure~\ref{fig:actual_candidate}).
If $\phi_k$ is too large, convergence of the finite precision quadratic regularization algorithm cannot be ensured because there is no guarantee that the actual search direction $-\gkact{k}$ is a descent direction.

\subsection{Inexact Norm Computation}%
\label{subsec:norm}

In Algorithm~\ref{alg:qrfp}, it is necessary to compute $\|\gkfp{k}\|$,  $\|s_k\|$ and $\|\xkfp{k}\|$.
Due to rounding errors, norm computations are inexact.
Lemma~\ref{lem:norm_inexact} states an upper bound on the error on norm computation, assuming that the computation of the norm is done in the same FP format as the argument.

\begin{lemma}%
	\label{lem:norm_inexact}
	Let Assumption~\ref{ass:gamma_n_bound}, be satisfied.
	The error on the computation of the norm of $\xkfp{}\in \F^n$ in finite precision is bounded as
	\begin{equation*}
		|\fl(\|\xkfp{}\|) - \|\xkfp{}\|\,| \leq  \fl(\|\xkfp{}\|)\beta_{n+2},
	\end{equation*}
	where $\beta_{n} := \max\left(|\sqrt{1-\gamma_{n}}-1|,|\sqrt{1+\gamma_{n}}-1|\right)$, and $\beta_{n+2}\leq 1$.
\end{lemma}

\begin{proof}
	The error on the dot product computation can be modeled as~\cite{higham2002accuracy},
	\begin{equation*}
		\fl(\xkfp{}^T \xkfp{}) = \xkfp{1}^2(1+\delta)^n\displaystyle + \sum_{i=2}^n \xkfp{i}^2 (1+\delta)^{n+2-i}.
	\end{equation*}
	As a consequence, there exists $\vartheta_n$ such that
	\begin{equation*}
		\fl(\|\xkfp{}\|^2) = \|\xkfp{}\|^2(1+\vartheta_n), \quad |\vartheta_n|\leq \gamma_n.
	\end{equation*}
	Since the square root computation is exactly rounded~\cite{ieee754}, $\forall \hat{y} \in \mathbb{F}$, $\fl(\sqrt{\hat{y}}) = \sqrt{\hat{y}}(1+\delta)$ and it follows that
	\begin{equation*}
		\begin{array}{ll}
			\fl(\sqrt{\fl(\|\xkfp{}\|^2)}) & = \sqrt{\fl(\|\xkfp{}\|^2)}(1+\delta) \\
			& = \sqrt{\|\xkfp{}\|^2(1+\vartheta_n)}(1+\delta)\\
			& = \|\xkfp{}\|\sqrt{(1+\vartheta_n)(1+\delta)^2}.\\
		\end{array}
	\end{equation*}
	Furthermore, because \(\gamma_n = nu\),
	\begin{equation*}
		\begin{array}{ll}
			(1+\vartheta_n)(1+\delta)^2 &\leq (1+\gamma_n)(1+\gamma_2)\\
			& = 1+(n+2)u+2nu^2\\
			& = 1+\gamma_{n+2} +2nu^2.
		\end{array}
	\end{equation*}
	Since $u \ll 1$, $2nu^2 = o((n+2)u)$ and we ignore this last term.
	As a consequence, there exists $\vartheta_{n+2}$ such that,
	\begin{equation*}
		\fl(\sqrt{\fl(\|\xkfp{}\|^2)}) = \|\xkfp{}\|\sqrt{1+\vartheta_{n+2}}.
	\end{equation*}
	Due to Assumption~\ref{ass:gamma_n_bound}, $\fl(\sqrt{\fl(\|\xkfp{}\|^2)})\geq 0$.
	Therefore,
	\begin{equation*}
		\begin{array}{ll}
			|\fl(\|\xkfp{}\|)-\|\xkfp{}\|\,| & = \fl(\|\xkfp{}\|)\,|1-\sqrt{1+\vartheta_{n+2}}|\\
		\end{array}
	\end{equation*}
	and since $|1-\sqrt{1+\vartheta_{n+2}}|\leq \max\left(|1-\sqrt{1-\gamma_{n+2}}|,|1-\sqrt{1+\gamma_{n+2}}|\right)$,
	we finally have,
	\begin{equation*}
		|\fl(\|\xkfp{}\|)-\|\xkfp{}\|\,|  \leq \fl(\|\xkfp{}\|)\,\max\left(|1-\sqrt{1-\gamma_{n+2}}|,|1-\sqrt{1+\gamma_{n+2}}|\right). \tag*{\qed}
	\end{equation*}
	\renewcommand{\qed}{}
\end{proof}

\section{Multi-Precision Algorithm}

\subsection{Algorithm Description}%
\label{sec:MPR2}
Algorithm~\ref{alg:qrfp} is an adaptation of Algorithm~\ref{alg:qr} that implements strategies to deal with the errors due to finite precision computations details in Section~\ref{sec:finite precision}.
Its proof of convergence is given in appendix.

\begin{algorithm}
	\caption{\textbf{MPR2} : Multi-precision quadratic regularization algorithm}%
	\label{alg:qrfp}
	
	\begin{algorithmic}
		\State \textbf{Step 0: Initialization:} Initial point $\xkfp{0}$, initial value $\sigkfp{0}$, minimal value $\sigkfp{\min}$ for $\sigkfp{}$, final gradient accuracy $\epsilon$, formula for $\gamma_n$.
		Constant values $\eta_0$, $\eta_1$, $\eta_2$, $\gamma_1$, $\gamma_2$, $\gamma_3$, $\wgaugbound$ such that,
		\begin{equation}
			\label{eq:param_cond}
			0<\eta_1\leq \eta_2<1\quad 0<\gamma_1<1<  \gamma_2 \leq \gamma_3, \quad \eta_0<\tfrac{1}{2}\eta_1 \quad \eta_0+\dfrac{\wgaugbound}{2} \leq \tfrac{1}{2}(1-\eta_2)
		\end{equation}
		Set $k=0$, compute $\widehat{f}_0 = \ffp(x_0)$, select $\gpreck{0}$ and $\cpreck{0}$.
		\State \textbf{Step 1: Check for termination:} If $k = 0$ or $\rho_{k-1} \geq \eta_1$ (previous iteration successful), compute $\widehat{g}_k = g(\xkfp{k})$.
		Compute $\widehat{\|\widehat{g}_k\|} = fl(\|\widehat{g}_k\|)$.
		Terminate if \begin{equation}
			\label{eq:stop_crit}
			\widehat{\|\widehat{g}_k\|} \leq \dfrac{1}{1+\beta_{n+2}(\mepsk{g}{k})}\dfrac{\epsilon}{1+\eg(\xkfp{k})}
		\end{equation}
		%     \textbf{If exact sum $\xkfp{k}+\skfp{k}$: }Terminate if $\dfrac{u}{1-u}+\alpha_n\dfrac{\eg(\xkfp{k})}{1-u} > \wgaugbound$\\
		%     \red{With condition in~\cite{birgin2017worst}, terminate if $\dfrac{\eg(\xkfp{k})+u}{1-u}>1/\sigkfp{k}$}.
		\State \textbf{Step 2: Step calculation: } Compute $\skfp{k} = \fl(\widehat{g}_k/\sigkfp{k})$.\\
		Compute $\widehat{\phi}_k = \fl(\|\xkfp{k}\| / \|\skfp{k}\|)$.\\
		Define $\phi_k = \widehat{\phi}_k\dfrac{1+\beta_{n+2}(\mepsk{x}{k})}{1-\beta_{n+2}(\mepsk{g}{k})}(1+\mepsk{g}{k})$.\\
		Define $\umixk{k} = \mepsk{g}{k} + \mepsk{c}{k} + \mepsk{g}{k}\mepsk{c}{k}$.\\
		Define 
		\begin{equation}
			\label{eq:muk}
			\mu_{k} = \dfrac{\alpha_{n+1}(\mepsk{g}{k}) \eg(\xkfp{k})(1+\lambda_k) + \alpha_{n+1}(\mepsk{g}{k}) \lambda_k + \mepsk{g}{k}+ \gamma_{n+1}(\mepsk{g}{k})\alpha_{n+1}(\mepsk{g}{k})}{1-\mepsk{g}{k}}.
		\end{equation} If 
		\begin{equation}
			\label{eq:muk_cond}
			\mu_{k} \leq \wgaugbound
		\end{equation} is not satisfied: increase $\gpreck{k}$ and go to Step 1 (at ``compute $\gfp_k$'' step) or increase $\cpreck{k}$ and redefine $\umixk{k}$ and $\mu_{k}$.
		Terminates if no available FP formats $(\cpreck{k},\gpreck{k})$ ensure the inequality.\\
		
		Compute $\ckfp{k} = \fl(\xkfp{k}+\skfp{k})$ and cast $\ckfp{k}$ into the format of index $\cpreck{k}$.\\
		Compute approximated Taylor series decrease $\tsdifffp_k = \fl(\widehat{g}_k^T\skfp{k})$.\\
		
		\State \label{step:obj_eval} \textbf{Step 3: Evaluate the objective function: }
		Choose $\fpreck{k}\geq \cpreck{k}$ such that 
		\begin{equation}
			\label{eq:fprec_c_cond}
			\ef(\widehat{c}_{k})\leq \eta_0 \tsdifffp_k.
		\end{equation}
		Compute $f_k^+ = \ffp(\ckfp{k}).$
		Terminates if no such precision is available.\\
		If $\ef(\xkfp{k})> \eta_0 \tsdifffp_k$, choose $\fprecrecompk{k}>\fpreck{k-1}$ as new format index for $\ffp$ such that 
		\begin{equation}
			\label{eq:fprec_x_cond}
			\ef(\widehat{x}_{k})\leq \eta_0 \tsdifffp_k.
		\end{equation}
		Re-compute $f_k = \ffp(\xkfp{k})$ with $\fprecrecompk{k}$.
		Terminates if no such precision is available.\\
		
		\State \textbf{Step 4: Acceptance of the trial point: } Define the ratio
		\begin{equation*}
			\rho_k = \dfrac{\fkfp{k} - \fkcfp{k}}{\tsdifffp_k}.
		\end{equation*}
		If $\rho_k \geq \eta_1$, then $\xkfp{k+1} =\ckfp{k}$, $\fkfp{k+1} = \fkcfp{k}$.
		Otherwise, set $\xkfp{k+1} = \xkfp{k}$, $\fkfp{k+1} = \fkfp{k}$.\\
		Select $\gpreck{k+1} \geq \xpreck{k}$,  select $\cpreck{k+1}$.
		
		\State \textbf{Step 5: Regularization  parameter update: }
		\begin{equation}
			\label{eq:sigma_update}
			\sigkfp{k+1} \in \left\{
			\begin{array}{ll}
				[\max(\sigkfp{\min},\gamma_1\sigkfp{k}),\sigkfp{k}] & \text{if } \rho_k \geq \eta_2\\
				
				[\sigkfp{k},\gamma_2\sigkfp{k}] & \text{if }\rho_k \in [\eta_1,\eta_2) \\
				
				[\gamma_2\sigkfp{k},\gamma_3\sigkfp{k}] & \text{if } \rho_k < \eta_1 \\
			\end{array}
			\right.
		\end{equation}
		$k = k+1$, go to Step 1.
	\end{algorithmic}
\end{algorithm}

At the initialization step, there is no modification on how the parameters should be chosen.
The only change is that the initial FP format $\cpreck{0}$ of the candidate must be chosen.

At Step 1, the factor $1/(1+\beta_{n+2})$ is added in Condition~\eqref{eq:stop_crit} to take into account the error due to norm computation of the approximated gradient $\gfp_k$ (see Lemma~\ref{lem:norm_inexact}). The new condition ensures $\nabla f(x_k)\leq \epsilon$ despite the gradient evaluation and norm computation errors,
\begin{equation}
	\label{eq:stop_cond_1}
	\begin{array}{ll}
		\|\nabla f(\xkfp{k})\| & \leq \|\nabla f(\xkfp{k}) -\gkfp{k}\| + \|\gkfp{k}\|\\
		&\leq (1+\eg(\xkfp{k}))\|\gkfp{k}\|\\
		&\leq (1+\eg(\xkfp{k}))\left(| \,\|\gkfp{k}\| - \fl(\|\gkfp{k}\|)| + \fl(\|\gkfp{k}\|)\right)\\
		&\leq (1+\eg(\xkfp{k}))(1+\beta_{n+2}(\mepsk{g}{k}))\fl(\|\gkfp{k}\|)
	\end{array}
\end{equation}

At Step 2, the convergence condition $\eg(x_k)\leq \wgaugbound$ becomes $\mu_{k}\leq \wgaugbound$. $\mu_k$ accounts for the error on gradient evaluation, and the errors due to finite precision computations.
The coefficient $\phi_k$ takes into account the norm computation errors that occur when computing $\widehat{\phi}_k$, and is a guaranteed upper bound on the ratio $\|\xkfp{k}\|/\|\skfp{k}\|$.
Indeed, due to the norm computation error (see Subsection~\ref{subsec:norm}),
\begin{equation*}
	\dfrac{\|\xkfp{k}\|}{\|\skfp{k}\|}  \leq \dfrac{\fl(\|\xkfp{k}\|)}{\fl(\|\skfp{k}\|)} \dfrac{1+\beta_{n+2}(\mepsk{x}{k})}{1-\beta_{n+2}(\mepsk{g}{k})}
\end{equation*}
Since the format of $\fl(\|\xkfp{x}\|)$ is $\xpreck{k}$ and the one of $\fl(\|\skfp{x}\|)$ is $\gpreck{k}$, and $\gpreck{k}\geq\xpreck{k}$ (see introductory remarks), it follows that,
\begin{equation*}
	\begin{array}{ll}
		\dfrac{\|\xkfp{k}\|}{\|\skfp{k}\|}
		& \leq \fl\left(\dfrac{\fl(\|\xkfp{k}\|)}{\fl(\|\skfp{k}\|)}\right)(1+\mepsk{g}{k}) \dfrac{1+\beta_{n+2}(\mepsk{x}{k})}{1-\beta_{n+2}(\mepsk{g}{k})}\\
		& = \widehat{\phi}_k\dfrac{1+\beta_{n+2}(\mepsk{x}{k})}{1-\beta_{n+2}(\mepsk{g}{k})}(1+\mepsk{g}{k}) = \phi_k.
	\end{array}
\end{equation*}
The coefficient $\umixk{k}$ corresponds to the virtual rounding errors accounting for inexact computation of $\ckfp{k}$ and the casting of $\ckfp{k}$ into a lower FP format.
By lowering the FP format precision, one enables to evaluate the objective at iteration $k$ or the gradient at iteration $k+1$ (if k is successful: $\xkfp{k+1} = \ckfp{k}$) with a lower FP format (see details in Subsection~\ref{subsec:candidate}). 

The coefficient $\mu_k$ in Algorithm~\ref{alg:qrfp} aggregates most of the errors due to finite precision computation listed in Section~\ref{sec:finite precision}, and is the counterpart of $\eg$ in Algorithm~\ref{alg:qr}.
The coefficient $\gamma_{n+1}$ and $\alpha_{n+1}$ account for the inexact approximated model decrease (see Subsection~\ref{subsec:bet-decrease}).
The coefficient $\lambda_{k}$, given in~\eqref{eq:lambda_def}, includes $\phi_k$ the upper bound on $\|\xkfp{k}\| / \|\skfp{k}\|$ and takes into account the inexact candidate computation explained in Section~\ref{subsec:candidate}.
The denominator $1-\mepsk{g}{k}$ account for the difference in norm between $\gfp(x_k)$ and $\gd_k$ (see Lemma~\ref{lem:gd_nabla_diff}).
Note that if we assume that the computations are done with infinite precision, \emph{i.e.} $\umixk{k} = \mepsk{g}{k}=0$, one has $\mu_k = \omega(\xkfp{k})$ and retrieve the condition in Algorithm~\ref{alg:qr}.
As such, all the errors due to finite precision computations aggregated in $\mu_k$ can simply be interpreted as noise on the gradient degrading the search direction.
The condition $\mu_k \leq \kappa_\mu$ ensures the convergence of Algorithm~\ref{alg:qrfp}.
If $\mu_{k}>\wgaugbound$, one can lower $\mu_k$ by decreasing $\cpreck{k}$, and by extension lower $\umixk{k}$ and therefore $\mu_{k}$.
The second option is to lower $\gpreck{k}$ and to recompute the gradient at Step 1 with a higher precision in order to lower $\eg(\xkfp{k})$.

Step 3 remains similar in Algorithm~\ref{alg:qr}.
A stopping criterion is added if the needed objective function evaluation accuracy cannot be achieved with the provided FP formats.
This criterion is needed to ensure the convergence.

At Step 4, the only modification is the selection of the FP format for the gradient evaluation at the next step, and the FP format of the candidate.
Step 5 remains the same, since we suppose that no rounding errors occur when computing $\sigkfp{k+1}$.

\subsection{Note on \texorpdfstring{$\widehat{\rho_k}$}{rhok}}
As indicated in Section~\ref{subsec:notation}, we assume that $\rho_k$ is defined in Algorithm~\ref{alg:qrfp}, that is, $\rho_k$ is computed with infinite precision.
This assumption is reasonable if a FP format with a precision higher than the formats used in Algorithm~\ref{alg:qrfp} is used to compute $\rho_k$.
However, if such precision is not available, it is critical to take into account the rounding errors occurring when computing $\rho_k$.
Indeed, if $\fkfp{k}$ and $\fkcfp{k}$ are large and close to each other, a ``catastrophic cancellation'' might occur at the numerator, and the value computed for $\rho_k$ is not relevant.

The rounding error occurring is such that,
\begin{equation*}
	\widehat{\rho_k} = \fl\left(\dfrac{\fkcfp{k}-\fkcfp{k}}{\tsdifffp_k}\right)=\dfrac{\fkcfp{k}(1+\delta)^2-\fkcfp{k}(1+\delta)^2}{\tsdifffp_k},
\end{equation*}
with delta bounded as $|\delta|\leq \mepsk{\rho}{k} =  \min(\mepsk{f}{k},\mepsk{f}{k-1},\mepsk{g}{k})$.
The error between the exact value of $f$ and $\fkfp{k}(1+\delta)^2$ can be bounded as,
\begin{equation*}
	|f(\xkfp{k}) - \fkfp{k}(1+\theta_2)| \leq |f(\xkfp{k}) - \fkfp{k}| + |\fkfp{k} - \fkfp{k}(1+\theta_2)| \leq \ef(\xkfp{k}) + \fkfp{k}\gamma_2(\mepsk{\rho}{k}) ,
\end{equation*}  and similarly,
\begin{equation*}
	|f(\ckfp{k}) - \fkcfp{k}(1+\delta)^2|  \leq \ef(\ckfp{k})+|\fkcfp{k}|\gamma_2(\mepsk{\rho}{k}).
\end{equation*}
Therefore, the inexact computation of $\rho_k$ can be taken into account by adapting the conditions at Step 3 as,
\begin{equation*}
	\ef(\xkfp{k})+\gamma_2(\mepsk{\rho}{k})|\fkfp{k}| \leq \eta_0\tsdifffp_k, \quad \ef(\ckfp{k})+\gamma_2(\mepsk{\rho}{k})|\fkcfp{k}| \leq\eta_0\tsdifffp_k.
\end{equation*}

\section{Numerical Tests}
\label{sec:num_experiments}

MPR2 is implemented in the Julia MultiPrecisionR2.jl package~\cite{mpr2} as part of JuliaSmoothOptimizer (JSO) organisation~\cite{orban-siqueira-optimizationproblems-2021}.
This implementation is robust to many pitfalls generated by the use of multiple FP formats and in particular low-precision, low-range FP formats (underflow, overflow, memory allocation for $x$, $c$ and $s$, etc.).
The detailed implementation of MPR2 is out of the scope of this paper, and is not described here. 
The reader can refer to the MultiPrecisonR2.jl package for implementation details.

The numerical results presented below can be reproduced from the dedicated file of the package.~\footnote{https://github.com/JuliaSmoothOptimizers/MultiPrecisionR2/blob/mpr2\_paper/docs/src/mpr2\_numerical\_test.md}
The results are obtained by running MPR2 on a standard laptop that does not support half precision natively.
Therefore, the computational effort savings presented in Section~\ref{subsec:eff_imp} are estimations.
The numerical results are obtained over a set of 164 unconstrained problems implemented in OptimizationProblems.jl~\cite{orban-siqueira-optimizationproblems-2021} package, with dimensions ranging from 1 to 100 variables.
The FP formats employed are half, single and double.
Assumption~\ref{ass:gamma_n_bound} is satisfied for all the problems, with $u_{\max} = 1/2048$ for half precision.
The tolerance on the first-order criterion is set to $\epsilon = \sqrt{\epsilon_M(\texttt{Float64})} \approx 1.5e^{-8}$.
The maximum number of iteration of MPR2 is set to 10,000.
All the defined values are computed with quadruple precision.

\subsection{Precision Selection Strategy}

The strategy implemented for selecting the evaluation FP formats and $\pi_c$ are the following.

At Step 2 of Algorithm~\ref{alg:qrfp}, the index $\gpreck{k}$ is selected as $\gpreck{k} = \cpreck{k}$. If $\mu_k>\kappa_\mu$, one can update either $\cpreck{k}$ to decrease $\umixk{k}$ or $\gpreck{k}$ to decrease $\eg(\xkfp{k})$ on which $\mu_k$ depends.
The update rule is 
\begin{equation*}
	\left\{
	\begin{array}{ll}
		\cpreck{k} = \cpreck{k} +1 & \text{ if }\cpreck{k}<\gpreck{k},\\
		\gpreck{k} = \gpreck{k} +1 & \text{otherwise},
	\end{array}
	\right.
\end{equation*} 
until $\mu_k<\kappa_\mu$.

At Step 3 of  Algorithm~\ref{alg:qrfp}, the index $\fpreck{k}$ for objective function evaluation at $\ckfp{k}$ is chosen as the index of the least accurate FP format whose unit roundoff $\mepsk{f}{k}$ satisfies
\begin{equation*}
	\ef(\xkfp{k}) \dfrac{\ffp(\xkfp{k})-\tsdifffp_k}{\ffp(\xkfp{k})} \dfrac{\mepsk{k}{f}}{\mepsk{k-1}{f}} \leq \eta_0\tsdifffp_k.
\end{equation*}
The left-hand side is a prediction of $\ef(\ckfp{k})$ based on the assumptions that $\ffp(\ckfp{k}) \approx \ffp(\xkfp{k}) - \tsdifffp_k$ and that $\ef$ grows linearly with the unit roundoff and is proportional to $f$.
These two assumptions are consistent with our observations over the test set.
If the convergence condition~\eqref{eq:fprec_c_cond} is not satisfied with the above selection strategy, $\fpreck{k}$ is increased by 1 until~\eqref{eq:fprec_c_cond} is ensured.

If it is necessary to re-evaluate $\ffp(\xkfp{k})$ to decrease $\ef(\xkfp{k})$ and enforce satisfy~\eqref{eq:fprec_x_cond}, the re-evaluation index $\fprecrecompk{k}$ is selected as the index of least accurate FP format whose unit roundoff $\mepsk{f^-}{k}$ satisfies
\begin{equation*}
	\ef(\xkfp{k})\dfrac{\mepsk{f^-}{k}}{\mepsk{f}{k-1}} \leq \eta_0\tsdifffp_k.
\end{equation*}
The left-hand side is a prediction of $\ef(\xkfp{k})$ evaluated with $\mepsk{f^-}{k}$ based on the assumption that $\ef(\xkfp{k})$ is linearly dependent on the unit roundoff.
This assumption is corroborated by our observations over the test set of problems.
If~\eqref{eq:fprec_x_cond} is still not satisfied with the above selection strategy, $\fprecrecompk{k}$ is increased by 1 until~\eqref{eq:fprec_x_cond} is satisfied.

At Step 4 of Algorithm~\ref{alg:qrfp}, the index $\cpreck{k}$ is chosen as $\cpreck{c} = \max(1,\fpreck{k}-1)$. This enables to lower the precision of the candidate and, by extension, to allow lower precision evaluation of the objective function and the gradient (see details in Section~\ref{subsec:candidate}).

\subsection{Guaranteed Implementation}

One challenge with Algorithm~\ref{alg:qrfp} is that the error bounds  $\omega_f$ and $\omega_g$ on objective and gradient must be available.
In the general case, obtaining analytical expression of these bounds might not be possible.
That is why the standard, or guaranteed, implementation of Algorithm~\ref{alg:qrfp} relies on interval arithmetic provided by the  JuliaIntervals.jl library~\cite{juliaintervals} to compute the error bounds $\omega_f$ and $\omega_g$.
Interval arithmetic enables to compute guaranteed bounds on the exact values (which cannot be computed in finite-precision) of the objective and the gradient, by accounting for the worst possible case of rounding errors accumulating during the evaluations.
From these bounds, it is possible to derive $\omega_f$ and $\omega_g$.
The results presented below are obtained with the guaranteed implementation of Algorithm~\ref{alg:qrfp}, using interval arithmetic.

%\red{ 
	%TO KEEP?
	%JuliaIntervals enables to compute guaranteed enclosures of the objective function and the gradient as
	%\begin{equation}
	%  \begin{array}{ll}
		%    f(x) \in \mathbf{f}(x) = [\underline{f(x)},\overline{f(x)}], \\
		%    \nabla f(x) \in \mathbf{\nabla f}(x) = [ \underline{\dfrac{\partial f}{\partial x_1}(x)},\overline{\dfrac{\partial f}{\partial x_1}(x)} ]\times \cdots \times [ \underline{\dfrac{\partial f}{\partial x_n}(x)},\overline{\dfrac{\partial f}{\partial x_n}}(x) ].
		%  \end{array}
	%\end{equation}
	%Selecting $\ffp$ and $\gfp$ as the midpoints of the intervals, one can derive bounds on evaluation error as
	%\begin{equation}
	%  \ffp(x) = mid(\mathbf{f}(x)) = \dfrac{\underline{f(x)}+\overline{f(x)}}{2}, \quad \omega_f(x) = diam(\mathbf{f}(x)) = \dfrac{\overline{f(x)}-\underline{f(x)}}{2},
	%\end{equation}
	%and
	%\begin{equation}
	%  \gfp(x) = mid(\mathbf{\nabla f}(x)),\quad \eg(x) = \dfrac{||diam(\mathbf{\nabla f}(x))||}{||\gfp(x)||} \dfrac{1+\beta_{n+2}}{1-\beta_{n+2}}.
	%\end{equation} 
	%where the $mid$ and $diam$ are extended component-wisely.
	%}

Table~\ref{tab:algo_param} lists the values of the parameters of Algorithm~\ref{alg:qrfp}.
\begin{table}[h]
	\begin{center}
		\begin{minipage}{\linewidth}
			\caption{MR2 parameters values}\label{tab:algo_param}%
			\begin{tabular}{@{}lllllll@{}}
				\toprule
				$\gamma_1$& $\gamma_2$ & $\gamma_3$ & $\eta_0$ &  $\eta_1$ & $\eta_2$ & $\wgaugbound$ \\
				1/2 & 1 & 2 & 0.05 & 0.1 & 0.7 & 0.2\\
				\bottomrule
			\end{tabular}
		\end{minipage}
	\end{center}
\end{table}
Over the test problems, the number of evaluations and \emph{successful} evaluations are computed.
An evaluation is successful if there is no need to re-evaluate the objective or the gradient with a higher precision FP format, and gives a metric to evaluate the precision selection strategy detailed above.
Table~\ref{tab:num_results} displays the number of problems for which Algorithm~\ref{alg:qrfp} reaches a first order critical point (FO), reaches the maximum number of iteration (MI), and for which the convergence conditions at Step 2 and 3 are not satisfied (F).
Note that 18 problems out of the 164 were ignored due to operators not supported by JuliaInterval.jl or other issues linked to this package.
Table~\ref{tab:num_results} also displays the percentage of evaluations performed in a given FP format for the gradient and the objective, and the success rate for each of the FP.
The success rate is the ratio between the number of successful evaluations and the total number of evaluations in a given format.

\begin{table}[h]
	\begin{center}
		\begin{minipage}{\linewidth}
			\caption{Percentage of total evaluations of objective and gradient for the given FP formats}\label{tab:num_results}%
			\begin{tabular}{@{}lllllllll@{}}
				\toprule
				FO & MI & F & \multicolumn{3}{c}{\% of obj.\ eval.\ (success rate in \%)} & \multicolumn{3}{c}{\% of grad.\ eval.\ (success rate in \%)}\\
				& & &16 bits & 32 bits & 64 bits & 16 bits & 32 bits & 64 bits\\
				\midrule
				46 & 65 & 35 & 0.1 (48.5) & 28.3 (91) & 71.6 (100) & 1.1 (87.6) & 79.5 (100) & 19.4 (100)\\
				\bottomrule
			\end{tabular}
		\end{minipage}
	\end{center}
\end{table}
The results in Table~\ref{tab:num_results} show that Float64 does not offer sufficient precision for objective or gradient evaluations to satisfy the convergence conditions at Step 2 and 3.

Table~\ref{tab:num_results} also shows that the objective function is evaluated with more precision than the gradient.
The reason is that the error bounds provided by interval arithmetic are, to a certain extent, proportional to the exact value.
Since the gradient tends to zero over the iterations, $\eg$ decreases accordingly and 32 bits offers in most case sufficient precision.
However, the value of the objective function at a critical point is not 0 in the general case, and the convergence conditions at Step 3 requires the error on the objective to be lower than $\eta_0 \tsdifffp_k$ which is proportional to the square of the norm of the gradient.
As a consequence, the highest precision is employed to evaluate the objective, especially for the last iterations where $\tsdifffp_k$ is small.

16 bits evaluations are marginal. 
This is due to the fact that overflow often occurs, or the precision is not sufficient for objective evaluation which can happens after only few iterations if the model decrease gets too small (see conditions~\eqref{eq:fprec_c_cond} and~\eqref{eq:fprec_x_cond}).
For gradient evaluation, the unit roundoff of 16 bits (1/2048) limits its use to value of $\phi_k$ lower than 2048 (see $\mu_k$ formula~\eqref{eq:muk}).
However, this value of $\phi_k$ can be reached in few iterations, after which only small step sizes (compare to the incumbent $\xkfp{k}$) can ensure sufficient objective decrease.

\subsection{Efficient Implementation}
\label{subsec:eff_imp}

As highlighted by the numerical results displayed in Table~\ref{tab:num_results}, the guaranteed implementation of MPR2 as described by Algorithm~\ref{alg:qrfp} stops because of lack of precision (F) on a large proportion of problems.
This is due to the combined facts that: 1. Algorithm~\ref{alg:qrfp} stops if the error bounds ($\omega_f$ or $\mu$) are too large and 2. these error bounds account for the worst-case of rounding error.
In practice, it is unlikely that these error bounds, although guaranteed, are reliable estimates of the actual evaluation errors.
Furthermore, interval evaluation is highly time consuming compared with classical evaluation, and makes its use prohibitive for many problems.
That is why we propose \pmprt, a relaxed and efficient version of Algorithm~\ref{alg:qrfp} that aims to increase the use of low-precision FP formats and reach convergence to a first-order critical point based on simple evaluation error models.
\pmprt has the following features:
\begin{itemize}
	\item the objective and gradient errors are estimated as
	\begin{equation}
		\label{eq:rel_model}
		\omega_f(x_k) = |\ffp(x_k)|2u, \quad \omega_g(x_k) = 2u,
	\end{equation}
	with $u$ the corresponding unit roundoff associated to $\fpreck{k}$ or $\gpreck{k}$,
	%\item the dot product error $\gamma_n$ is chosen as $\sqrt{n}u$, which reduces the model decrease error,
	\item \pmprt does not stop if the conditions on $\mu_k$ and $\omega_f(\xkfp{k})$, $\omega_f(\ckfp{k})$, are not met with the highest precision FP format.
\end{itemize}
We also propose to decrease $\mu$ by a factor $a \in ]0,1]$ to relax the condition~\eqref{eq:muk_cond} into $a\mu \leq \kappa_\mu$ to assess whether even more computational effort can be saved.
We compare \pmprt run with Float16, Float32 and Float64 with R2, the generic single precision version of MPR2 that does not take any finite-precision error into account, implemented in the JSOSolvers.jl package~\cite{jso_solver} run with Float64.
Both \pmprt and R2 use the parameters given in Table~\ref{tab:algo_param}.
For the sake of comparison, the stopping criterion for \pmprt is set to $\widehat{\|\widehat{g}_k\|} \leq \epsilon$, which is the one used in R2.
Table~\ref{tab:obj_res} (resp. Table~\ref{tab:grad_res}) displays the number of problem solved by \pmprt and R2 as well as the percentage of total objective function (resp. gradient) evaluations over the whole set of problems in each FP format with the associated success rate, and the estimated computational effort ratio between \pmprt and R2 for time and energy. The lower the ratio, the higher the effort savings with \pmprt.
The time and energy efforts are estimated from the following observation: dividing the computation precision by two divides the computation time by two and the energy consumption by four~\cite{galal2010energy}, \emph{e.g.}, evaluating the objective function in Float16 is four times faster and requires sixteen times less energy than evaluating it in Float64.

\begin{table}[h]
	\begin{center}
		\begin{minipage}{\linewidth}
			\caption{Estimated computational effort saved by \pmprt compared with R2 for objective evaluation.}\label{tab:obj_res}%
			\begin{tabular}{@{}lllllllll@{}}
				\toprule
				Algo & $a$ & \multicolumn{3}{c}{\% of obj.\ eval.\ (success rate in \%)} & time ratio & energy ratio & pb solved(/164) \\
				& & 16 bits & 32 bits & 64 bits & & & \\
				R2 & - & - &  - & 1.0(100) & 1.0 & 1.0 & 73 \\
				\pmprt & 1.0 & 0.8(92.7) & 54.5($>$99.9) & 44.6(100) & 0.633 & 0.512 & 67 \\
				\pmprt & 0.1 & 1.3(93) & 57($>$99.9) & 41.7(100) & 0.604 & 0.48 & 64\\
				\pmprt & 0.01 & 2.2(96.5) & 56.9($>$99.9) & 40.9(100) & 0.64 & 0.506 & 54\\
				\bottomrule
			\end{tabular}
		\end{minipage}
	\end{center}
\end{table} 

\begin{table}[h]
	\begin{center}
		\begin{minipage}{\linewidth}
			\caption{Estimated computational effort saved by \pmprt compared with R2 for gradient evaluation.}\label{tab:grad_res}%
			\begin{tabular}{@{}lllllllll@{}}
				\toprule
				Algo & $a$ & \multicolumn{3}{c}{\% of grad.\ eval.\ (success rate in \%)} & time ratio & energy ratio & pb solved(/164) \\
				& & 16 bits & 32 bits & 64 bits & & & \\
				R2 & - & - &  - & 1.0(100) & 1.0 & 1.0 & 73 \\
				\pmprt & 1.0 & 1.9(91.8) & 81.5($>$99.9) & 16.7(100) & 0.566 & 0.363 & 67\\
				\pmprt & 0.1 & 14.6(93.4) & 74.2($>$99.9) & 11.2(100) & 0.503 & 0.297 & 64\\
				\pmprt & 0.01 & 21(88.5) & 69.3($>$99.9) & 9.6(100) & 0.457 & 0.261 & 54\\
				\bottomrule
			\end{tabular}
		\end{minipage}
	\end{center}
\end{table} 

From Table~\ref{tab:obj_res} and Table~\ref{tab:grad_res}, it appears that \pmprt enables significant time and energy savings, but does not solve as many problems as R2.
The savings for gradient evaluations are inversely proportional to $a$, since smaller $a$ means larger values of $\mu$ are tolerated, and, by extension, larger gradient error $\omega_g$, permitting to employ lower precision FP formats.
The results also show that the proposed strategy for \pmprt has a high success rate, therefore avoiding unnecessary re-evaluation with a higher precision FP format.

The performance profiles comparing R2 and \pmprt with respect to time and energy effort for the objective and the gradient are displayed in Figure~\ref{fig:per_profile}.
The performance profiles show that \pmprt outperforms R2, in accordance to the results displayed in Table~\ref{tab:obj_res} and Table~\ref{tab:grad_res}.

\begin{figure}[htbp]
	\centering
	\subfigure[Objective time effort perfomance profile.]{\includegraphics[width=0.45\linewidth]{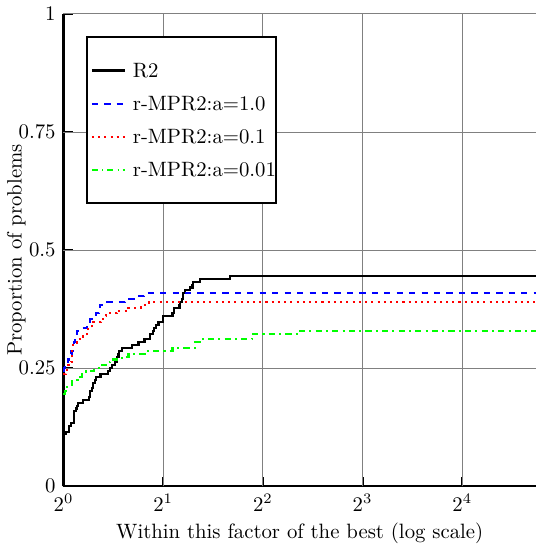}}
	\subfigure[Objective energy effort perfomance profile.]{\includegraphics[width=0.45\linewidth]{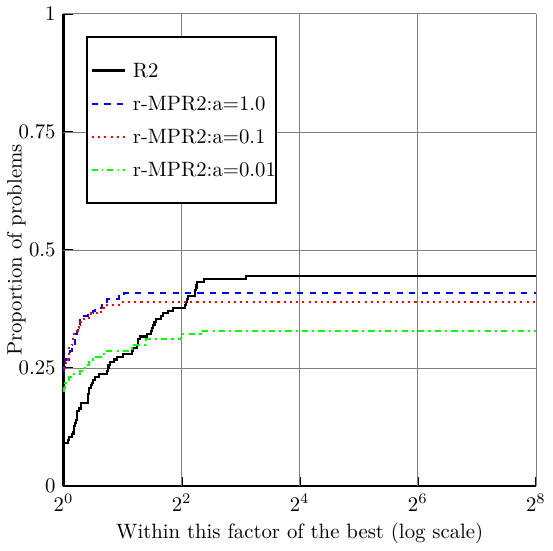}}
	\subfigure[Gradient time effort perfomance profile.]{\includegraphics[width=0.45\linewidth]{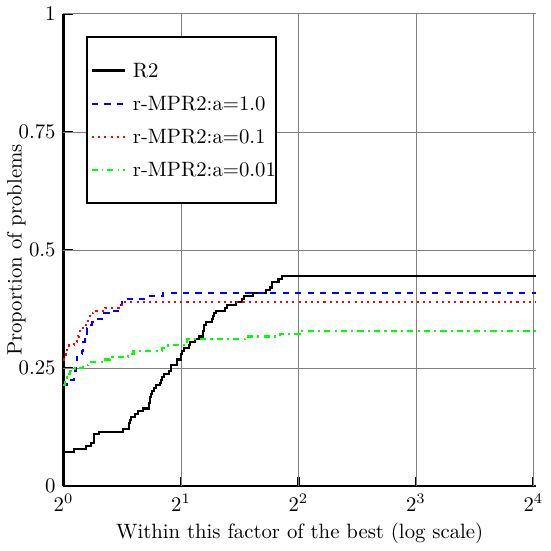}}
	\subfigure[Gradient energy effort perfomance profile.]{\includegraphics[width=0.45\linewidth]{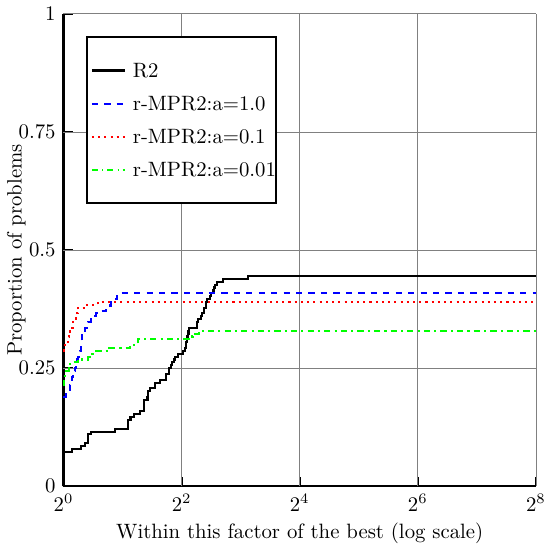}}
	\caption{Performance profiles comparing R2 and Algorithm~\ref{alg:qrfp} with respect to time and energy effort for objective and gradient.}
	\label{fig:per_profile}
\end{figure}

The limitation of \pmprt is that it does not solve as many problems as R2.
This is not necessarily intuitive: since \pmprt can perform evaluations with Float64 just as R2 and uses the same parameters (given in Table~\ref{tab:algo_param}), one could expect \pmprt to be able to solve all the problems that R2 solves.
From what we observed on the few problems that R2 solves and \pmprt does not, our interpretation is that the objective function error model is too optimistic, that is, the evaluation error is greater than what the model~\eqref{eq:rel_model} estimates.
This leads to computing $\ffp(\xkfp{k})$ and $\ffp(\ckfp{k})$ in FP formats that does not offer enough precision.
As a consequence, $\rho$ is not a reliable indicator for iteration success and can falsely be estimated lower than $\eta_1$ over several iterations, causing $\sigma$ to increase.
This is typically the case if $f(\xkfp{k}) > f(\ckfp{k})$ but $\fkfp{}(\xkfp{k}) < \fkfp{}(\ckfp{k})$, which case should be avoided thanks to conditions~\eqref{eq:fprec_x_cond} and~\eqref{eq:fprec_c_cond} and reliable bounds $\omega_f(\xkfp{k})$, $\omega_f(\ckfp{k})$.
In certain cases, $\xkfp{k}$ is a local minimum for the noisy function $\ffp$.
The iterates become ``trapped" in this neighbourhood since $\rho$ can only be negative and $\sigma$ increases, until $\ckfp{k} = \xkfp{k}$ because of candidate computation underflow.
Note that such behaviour does happen with R2, but not on as many problems since employing only double precision helps avoid unreliable values of $\rho$.

One strategy to avoid this shortcoming would be to detect when the candidate is ``trapped" and from then accept all the steps and not rely on the unreliable value of $\rho$.
This however raises the question of the relevance of trying to reach a first-order point in which neighbourhood the noise level on the objective evaluation makes it impossible to compare two close solutions. 

\section{Conclusion}
\label{sec:conclusion}

In this paper, we proposed MPR2 (Algorithm~\ref{alg:qrfp}), a multi-precision adaptation of R2 algorithm. 
We performed a comprehensive convergence analysis of Algorithm~\ref{alg:qrfp} that takes not only the objective function and gradient evaluation errors due to FP computations, but all the rounding errors (step and candidate computations, etc.) expect for few defined variables that involve few operators.
This convergence analysis provided convergence conditions that takes the rounding error into account, aggregated by the $\mu_k$ value.

We implemented rigorously Algorithm~\ref{alg:qrfp}, using interval analysis to guarantee the evaluation errors of the objective function and the gradient.
The numerical tests showed that the majority of the evaluations can be performed in Float32 on a bank of problems ranging between 1 and 100 variables, while Float16 is scarcely used.
It also appeared that Float64 does not offer sufficient accuracy to guarantee the convergence of Algorithm~\ref{alg:qrfp} to a first order critical point. 

We also proposed \pmprt, a version of Algorithm~\ref{alg:qrfp} based on relative error model for the objective and gradient evaluation, avoiding the time consuming use of interval evaluation.
The numerical results show that \pmprt enables significant computational savings compared with the original R2 algorithm, at the cost of lower robustness.

Future work could include extending the proposed analysis to second order methods.
Such an analysis is challenging since the step is not simply computed from the gradient but from a process minimizing a quadratic subproblem.
Such process (e.g. conjugate gradient) accumulates rounding errors that have to be taken into account.

\appendix

\section{Convergence Analysis}%
\label{sec:proof}

\begin{assumption}
	There exists $f_{low}$ such that,
	\begin{equation}
		\tag{AS.1}
		\label{as:flb}
		\forall x \in \R^n,\, f(x) \geq f_{low}.
	\end{equation} 
\end{assumption}

\begin{lemma}%
	\label{lem:f_pseudotaylor_diff_sum_inexact}
	For all $k>0$,
	\begin{equation}
		\label{eq:f_pseudotaylor_diff_sum_inexact}
		\begin{array}{ll}
			|\ffp(\xkfp{k}) - \ffp(\ckfp{k}) -\tsdifffp_k| \leq & 2\eta_0\tsdifffp_k + \tfrac{1}{2}L\|\skact{k}\|^2
			+ \lambda_k\|\nabla f(\xkfp{k})\|\,\|\skfp{k}\|\\ & +  |\gd_k^T\skfp{k}(1+\vartheta_{n+1}) - \nabla f(\xkfp{k})^T\skfp{k}|.
		\end{array}
	\end{equation}
\end{lemma}

\begin{proof}
	With the triangular inequality, one has,
	\begin{equation*}
		\begin{array}{ll}
			|\ffp(\xkfp{k}) - \ffp(\ckfp{k}) -\tsdifffp_k| & \leq |\ffp(\xkfp{k})-f(\xkfp{k})| + |\ffp(\xkfp{k})-f(\xkfp{k})| + |f(\xkfp{k}) - f(\widehat{c}_k) - \tsdifffp_k|.
		\end{array}
	\end{equation*}
	The termination condition at Step 3 and~\eqref{eq:fgerr} enables to bound, we derive
	\begin{equation}
		\label{eq:f_pseudotaylor_diff_sum_inexact_proof0}
		\begin{array}{ll}
			|\ffp(\xkfp{k}) - \ffp(\ckfp{k}) -\tsdifffp_k| & \leq 2\eta_0\tsdifffp_k + |f(\xkfp{k}) - f(\widehat{c}_k) - \tsdifffp_k|.
		\end{array}
	\end{equation}
	By Lemma~\ref{lem:pseudo_reduction}, the second term of the right-hand side of~\eqref{eq:f_pseudotaylor_diff_sum_inexact_proof0} can be bounded as,
	\begin{equation}
		\label{eq:f_pseudotaylor_diff_sum_inexact_proof1}
		\begin{array}{ll}
			& |f(\xkfp{k}) - f(\widehat{c}_k) - \tsdifffp_k|\\
			= &  |f(\xkfp{k}) - f(\widehat{c}_k) + \gd_k^T\skfp{k}(1+\vartheta_{n+1})|\\
			\leq & |f(\xkfp{k}) - f(\widehat{c}_k) + \nabla f(\xkfp{k})^T\skfp{k}| + |\gd_k^T\skfp{k}(1+\vartheta_{n+1}) - \nabla f(\xkfp{k})^T\skfp{k}|.
		\end{array}
	\end{equation}
	
	By triangular inequality and with~\eqref{eq:actual_candidate}, the first term of the right-hand side of~\eqref{eq:f_pseudotaylor_diff_sum_inexact_proof1} can be bounded as,
	\begin{equation}
		\label{eq:f_pseudotaylor_diff_sum_inexact_proof_2}
		\begin{array}{ll}
			|f(\xkfp{k}) - f(\widehat{c}_k) + \nabla f(\xkfp{k})^T\skfp{k}|
			\leq & |f(\xkfp{k}) - f(\widehat{c}_k) + \nabla f(\xkfp{k})^T\skact{k}| \\
			& + | \nabla f(\xkfp{k})^T\skfp{k} - \nabla f(\xkfp{k})^T\skact{k}|\\
		\end{array}
	\end{equation}
	With~\eqref{eq:ft_diff} one further has 
	\begin{equation}
		\label{eq:f_pseudotaylor_diff_sum_inexact_proof_4}
		\begin{array}{ll}
			|f(\xkfp{k}) - f(\ckfp{k}) + \nabla f(\xkfp{k})^T\skact{k}| & = |T(\skact{k})- f(\widehat{c}_k)|\\
			& = |T(\skact{k})- f(\xkfp{k}+\skact{k})|\\
			& \leq \tfrac{1}{2}L\|\skact{k}\|^2.
		\end{array}
	\end{equation}
	By triangular inequality and with~\eqref{eq:phi_k_def} and~\eqref{eq:lambda_def},
	\begin{equation}
		\label{eq:f_pseudotaylor_diff_sum_inexact_proof_3}
		\begin{array}{ll}
			|\nabla f(\xkfp{k})^T\skfp{k} - \nabla f(\xkfp{k})^T(\xkfp{k} \dmixk{k} +\skfp{k}(1+\dmixk{k}))| & = |\dmixk{k}\nabla f(\xkfp{k})^T(\xkfp{k}+\skfp{k})|\\
			& \leq \umixk{k}|\nabla f(\xkfp{k})^T(\xkfp{k}+\skfp{k})|\\
			& \leq \umixk{k}\|\nabla f(\xkfp{k})\|\,\|\xkfp{k}+\skfp{k}\|\\
			& \leq \umixk{k}\|\nabla f(\xkfp{k})\|\,(\|\xkfp{k}\|+\|\skfp{k}\|)\\
			& \leq \umixk{k}\|\nabla f(\xkfp{k})\|\,(\phi_k\|\skfp{k}\|+\|\skfp{k}\|)\\
			& \leq \lambda_{k}\|\nabla f(\xkfp{k})\|\,\|\skfp{k}\|.
		\end{array}
	\end{equation}
	
	Injecting Inequalities~\eqref{eq:f_pseudotaylor_diff_sum_inexact_proof_4} and~\eqref{eq:f_pseudotaylor_diff_sum_inexact_proof_3} into~\eqref{eq:f_pseudotaylor_diff_sum_inexact_proof_2} yields,
	\begin{equation}
		\label{eq:f_pseudotaylor_diff_sum_inexact_proof5}
		|f(\xkfp{k}) - f(\widehat{c}_k) - \nabla f(\xkfp{k})^T\skfp{k}| \leq \tfrac{1}{2}L\|\skact{k}\|^2 + \lambda_k\|\nabla f(\xkfp{k})\|\,\|\skfp{k}\|
	\end{equation}

	Putting together Inequalities~\eqref{eq:f_pseudotaylor_diff_sum_inexact_proof0},~\eqref{eq:f_pseudotaylor_diff_sum_inexact_proof1} and~\eqref{eq:f_pseudotaylor_diff_sum_inexact_proof5} yields~\eqref{eq:f_pseudotaylor_diff_sum_inexact}.
\end{proof}

\begin{lemma}%
	\label{lem:gd_nabla_diff}
	For all $k>0$,
	\begin{equation}
		\label{eq:gd_nabla_diff}
		\left|\dfrac{\gd_k^T\skfp{k}(1+\vartheta_{n+1}) - \nabla f(\xkfp{k})^T\skfp{k}}{\tsdifffp_k}\right| \leq \dfrac{\mepsk{g}{k}+\gamma_{n+1}(\mepsk{g}{k})\alpha_{n+1}(\mepsk{g}{k})+\alpha_{n+1}(\mepsk{g}{k})\eg(\xkfp{k})}{1-\mepsk{g}{k}}.
	\end{equation}
\end{lemma}
\begin{proof}
	With Lemma~\ref{lem:pseudo_reduction} and triangular inequality,
	\begin{equation}
		\label{eq:sk_out}
		\begin{array}{ll}
			\left|\dfrac{\gd_k^T\skfp{k}(1+\vartheta_{n+1}) - \nabla f(\xkfp{k})^T\skfp{k}}{\tsdifffp_k}\right| 
			& = \left|\dfrac{\gd_k^T\skfp{k}(1+\vartheta_{n+1}) - \nabla f(\xkfp{k})^T\skfp{k}}{\|\gd_k\|^2/\sigkfp{k}(1+\vartheta_{n+1})}\right|\\ 
			& \leq \dfrac{\|\gd_k(1+\vartheta_{n+1}) - \nabla f(\xkfp{k})\| \, \|\skfp{k}\|}{\|\gd_k\|^2/\sigkfp{k}(1+\vartheta_{n+1})}\\
		\end{array}
	\end{equation}
	Note that in~\eqref{eq:sk_out}, we implicitly use the equality \begin{equation*}
		\gd_k^T\skfp{k}(1+\vartheta_{n+1}) = \left(\gd_k\odot(1+\vartheta_{n+1},\dots,\vartheta_{n+1})\right)^T \skfp{k}
	\end{equation*}  and with the abusive notation
	$ \gd_k\odot(1+\vartheta_{n+1},\dots,\vartheta_{n+1})^T$ rewrites $\gd_k(1+\vartheta_{n+1})$.
	
	Recalling that $\|\skfp{k}\| = \|\gd_k\|/\sigkfp{k}$, 
	we further derive from~\eqref{eq:sk_out},
	\begin{equation}
		\label{eq:gd_nabla_diff_proof_1}
		\begin{array}{ll}
			\left|\dfrac{\gd_k^T\skfp{k}(1+\vartheta_{n+1}) - \nabla f(\xkfp{k})^T\skfp{k}}{\tsdifffp_k}\right|
			& \leq \dfrac{\|\gd_k(1+\vartheta_{n+1}) - \nabla f(\xkfp{k})\|}{\|\gd_k\|(1+\vartheta_{n+1})}\\
			& \leq \dfrac{\|\gd_k(1+\vartheta_{n+1}) - \widehat{g}_k\|}{\|\gd_k\|(1+\vartheta_{n+1})} + \dfrac{\|\widehat{g}_k - \nabla f(\xkfp{k})\|}{\|\gd_k\|(1+\vartheta_{n+1})}\\
		\end{array}
	\end{equation}
	With the triangular inequality, Lemma~\ref{lem:g_actual_diff} and the bound~\eqref{eq:alphan}, the first term of the right-hand side of~\eqref{eq:gd_nabla_diff_proof_1} can be bounded as,
	\begin{equation}
		\label{eq:error_out}
		\begin{array}{ll}
			\dfrac{\|\gd_k(1+\vartheta_{n+1}) - \widehat{g}_k\|}{\|\gd_k\|(1+\vartheta_{n+1})} 
			& = \dfrac{\|\gd_k(1+\vartheta_{n+1}) -\widehat{g}_k - \widehat{g}_k(1+\vartheta_{n+1}) +\widehat{g}_k(1+\vartheta_{n+1})\|}{\|\gd_k\|(1+\vartheta_{n+1})} \\
			& \leq \dfrac{\|\gd_k(1+\vartheta_{n+1}) - \widehat{g}_k(1+\vartheta_{n+1})\|}{\|\gd_k(1+\vartheta_{n+1})\|} + \dfrac{\|\widehat{g}_k(1+\vartheta_{n+1}) - \widehat{g}_k\|}{\|\gd_k\|(1+\vartheta_{n+1})} \\
			&  = \dfrac{\|\gd_k - \widehat{g}_k\|}{\|\gd_k\|} + \dfrac{\|\widehat{g}_k\|\,|\vartheta_{n+1}|}{\|\gd_k\|(1+\vartheta_{n+1})} \\
			&  \leq \dfrac{\|\gd_k - \widehat{g}_k\|}{\|\gd_k\|} + \gamma_{n+1}(\mepsk{g}{k})\alpha_{n+1}(\mepsk{g}{k})\dfrac{\|\widehat{g}_k\|}{\|\gd_k\|}\\
		\end{array}
	\end{equation}
	By Lemma~\ref{lem:g_actual_diff}, we further derive from~\eqref{eq:error_out},
	\begin{equation}
		\label{eq:gd_nabla_diff_proof_3}
		\begin{array}{ll}
			\dfrac{\|\gd_k(1+\vartheta_{n+1}) - \widehat{g}_k\|}{\|\gd_k\|(1+\vartheta_{n+1})} & 
			\leq \mepsk{g}{k}\dfrac{\|\gkfp{k}\|}{\|\gd_k\|} + \dfrac{\alpha_{n+1}(\mepsk{g}{k})\gamma_{n+1}(\mepsk{g}{k})}{1-\mepsk{g}{k}}\\
			& 
			\leq \dfrac{\mepsk{g}{k}+\alpha_{n+1}(\mepsk{g}{k})\gamma_{n+1}(\mepsk{g}{k})}{1-\mepsk{g}{k}}.
		\end{array}
	\end{equation}
	Using~\eqref{eq:alphan} to bound $1/(1+\vartheta_{n+1})$, and with Lemma~\ref{lem:g_actual_diff}, the second term in the right-hand side of~\eqref{eq:gd_nabla_diff_proof_1} can be bounded as,
	\begin{equation}
		\label{eq:gd_nabla_diff_proof_2}
		\begin{array}{ll}
			\dfrac{\|\widehat{g}_k - \nabla f(\xkfp{k})\|}{\|\gd_k\|(1+\vartheta_{n+1})} & \leq \alpha_{n+1}(\mepsk{g}{k})\dfrac{\eg(\xkfp{k})\|\widehat{g}_k\|}{\|\gd_k\|}\\
			& \leq \alpha_{n+1}(\mepsk{g}{k})\dfrac{\eg(\xkfp{k})}{1-\mepsk{g}{k}}.
		\end{array}
	\end{equation}
	Putting Inequalities~\eqref{eq:gd_nabla_diff_proof_3} and~\eqref{eq:gd_nabla_diff_proof_2} in Inequality~\eqref{eq:gd_nabla_diff_proof_1}, one obtains~\eqref{eq:gd_nabla_diff}.
\end{proof}

\begin{lemma}%
	\label{lem:sigma-very-successful-inexact-sum}
	For all $k>0$,
	\begin{equation}
		\label{eq:sigma-very-successful-inexact-sum}
		\dfrac{1}{\sigkfp{k}} \leq \left[1-\eta_2 - \eta_0 - \dfrac{\wgaugbound}{2}\right]\dfrac{1}{\alpha_{n+1}(\mepsk{g}{k})L(1+\lambda_k)^2} \implies \rho_k\geq \eta_2.
	\end{equation} 
	
\end{lemma}

\begin{proof}
	With Lemma~\ref{lem:f_pseudotaylor_diff_sum_inexact},
	\begin{equation*}
		\begin{array}{ll}
			|\rho_k - 1|  & = \left| \dfrac{\ffp(\xkfp{k}) - \ffp(\ckfp{k}) - \tsdifffp_k}{ \tsdifffp_k}\right|\\ 
			& \leq
			2\eta_0 + \dfrac{\tfrac{1}{2}L\|\skact{k}\|^2
				+ \lambda_k\|\nabla f(\xkfp{k})\|\,\|\skfp{k}\| +  |\gd_k^T\skfp{k}(1+\vartheta_{n+1}) - \nabla f(\xkfp{k})^T\skfp{k}|}{\tsdifffp_k},
		\end{array}
	\end{equation*} 
	and with Lemma Lemma~\ref{lem:gd_nabla_diff} we further derive
	\begin{equation}
		\label{eq:sigma-very-successful-inexact-sum_proof_1}
		\begin{array}{ll}
			& |\rho_k - 1|\\ 
			\leq &
			2\eta_0 + \dfrac{\tfrac{1}{2}L\|\skact{k}\|^2}{\tsdifffp_k}
			+ \dfrac{\lambda_k\|\nabla f(\xkfp{k})\|\,\|\skfp{k}\|}{\tsdifffp_k} + \dfrac{\mepsk{g}{k}+\gamma_{n+1}(\mepsk{g}{k})\alpha_{n+1}(\mepsk{g}{k})+\alpha_{n+1}(\mepsk{g}{k})\eg(\xkfp{k})}{1-\mepsk{g}{k}}.
		\end{array}
	\end{equation} 
	With Lemma~\ref{lem:pseudo_reduction}, Lemma~\ref{lem:actual_step_norm} and~\eqref{eq:alphan}, and recalling that $\|\skfp{k}\| = \|\gd_k\|/\sigkfp{k}$, the second term of the right-hand side of~\eqref{eq:sigma-very-successful-inexact-sum_proof_1} can be bounded as,
	\begin{equation}
		\label{eq:sigma-very-successful-inexact-sum_proof_2}
		\begin{array}{ll}
			\dfrac{\tfrac{1}{2}L\|\skact{k}\|^2}{\tsdifffp_k} & = \dfrac{\tfrac{1}{2}L\|\skact{k}\|^2}{\|\gd_k\|^2/\sigkfp{k}(1+\vartheta_{n+1})} \\
			& \leq \alpha_{n+1}(\mepsk{g}{k}) \dfrac{\tfrac{1}{2}L\|\skfp{k}\|^2(1+\lambda_k)^2}{\|\gd_k\|^2/\sigkfp{k}}\\
			& = \alpha_{n+1}(\mepsk{g}{k}) \dfrac{\tfrac{1}{2}L(1+\lambda_{k})^2}{\sigkfp{k}}.
		\end{array}
	\end{equation}
	with $\lambda_k$ as in~\eqref{eq:lambda_def}.
	
	Similarly, with Lemma~\ref{lem:pseudo_reduction} and~\eqref{eq:alphan} the third term of the right-hand side of~\eqref{eq:sigma-very-successful-inexact-sum_proof_1} can be bounded as
	\begin{equation*}
		\begin{array}{ll}
			\dfrac{\lambda_k\|\nabla f(\xkfp{k})\|\,\|\skfp{k}\|}{\tsdifffp_k} & \leq \alpha_{n+1} \dfrac{\lambda_k\|\nabla f(\xkfp{k})\|\,\|\skfp{k}\|}{\|\gd_k\|^2/\sigkfp{k}},\\
		\end{array}
	\end{equation*}
	and by Lemma~\ref{lem:g_actual_diff} and triangular inequality,
	\begin{equation}
		\label{eq:sigma-very-successful-inexact-sum_proof_3}
		\begin{array}{ll}
			\dfrac{\lambda_k\|\nabla f(\xkfp{k})\|\,\|\skfp{k}\|}{\tsdifffp_k} & \leq \alpha_{n+1}(\mepsk{g}{k}) \lambda_k\left( \dfrac{\|\nabla f(\xkfp{k})-\widehat{g}_k\|}{\|\gd_k\|} +  \dfrac{\|\widehat{g}_k\|}{\|\gd_k\|}\right)\\
			& \leq \alpha_{n+1}(\mepsk{g}{k}) \lambda_k\left( \dfrac{\eg(\xkfp{k})}{1-\mepsk{g}{k}} +  \dfrac{1}{1-\mepsk{g}{k}}\right).\\
		\end{array}
	\end{equation}
	Putting Inequalities~\eqref{eq:sigma-very-successful-inexact-sum_proof_1},~\eqref{eq:sigma-very-successful-inexact-sum_proof_2} and~\eqref{eq:sigma-very-successful-inexact-sum_proof_3} together,
	\begin{equation}
		\label{eq:sigma-very-successful-inexact-sum_proof_4}
		\begin{array}{ll}
			|\rho_k - 1|  \leq 2 \eta_0  & + \alpha_{n+1}(\mepsk{g}{k}) \dfrac{\tfrac{1}{2}L(1+\lambda_k)^2}{\sigkfp{k}} + \\
			& \dfrac{\alpha_{n+1}(\mepsk{g}{k}) \eg(\xkfp{k})(1+\lambda_k) + \alpha_{n+1}(\mepsk{g}{k}) \lambda_k + \mepsk{g}{k}+ \gamma_{n+1}(\mepsk{g}{k})\alpha_{n+1}(\mepsk{g}{k})}{1-\mepsk{g}{k}}.
		\end{array}
	\end{equation} 
	
	Step 2 of Algorithm~\ref{alg:qrfp} ensures that the last term of the right-hand side of \eqref{eq:sigma-very-successful-inexact-sum_proof_4} is bounded as
	\begin{equation*}
		\mu_k = \dfrac{\alpha_{n+1}(\mepsk{g}{k}) \eg(\xkfp{k})(1+\lambda_k) + \alpha_{n+1}(\mepsk{g}{k})\lambda_k + \mepsk{g}{k}+ \gamma_{n+1}(\mepsk{g}{k})\alpha_{n+1}(\mepsk{g}{k})}{1-\mepsk{g}{k}} \leq \wgaugbound,
	\end{equation*}
	and it follows that,
	\begin{equation*}
		|\rho_k-1| \leq 2\eta_0+ \wgaugbound + \alpha_{n+1}(\mepsk{g}{k})\dfrac{\tfrac{1}{2}L(1+\lambda_k)^2}{\sigkfp{k}}.
	\end{equation*}
	Therefore, with $1/\sigkfp{k} \leq \left[1-\eta_2 - 2\eta_0 - \wgaugbound\right]\dfrac{1}{\alpha_{n+1}(\mepsk{g}{k})L(1+\lambda_k)^2}$ enforces by~\eqref{eq:sigma-very-successful-inexact-sum}, the inequality rewrites
	\begin{equation*}
		\begin{array}{ll}
			|\rho-1| & \leq \tfrac{1}{2}(1-\eta_2) + \eta_0 + \dfrac{\wgaugbound}{2}
		\end{array} 
	\end{equation*}
	and since the parameters are chosen such that $\eta_0+\dfrac{\wgaugbound}{2} \leq \tfrac{1}{2}(1-\eta_2)$, it follows that $|\rho-1| \leq 1-\eta_2$.
\end{proof}

\begin{lemma}%
	\label{lem:sigma_max}
	For all $k>0$,
	\begin{equation}
		\label{eq:sigma_max}
		\sigkfp{k} \leq \sigma_{\max} =  \dfrac{\gamma_3 L(1+\lambda_{\max})^2\alpha_{n+1}(\mepsk{g}{k})}{ 1-\eta_2 - \eta_0 - \dfrac{\wgaugbound}{2}},
	\end{equation}
	with \begin{equation*}
		\lambda_{\max} = \dfrac{\kappa_\mu}{\alpha_{n+1}(\mepsk{g}{k})}(1-\mepsmin).
	\end{equation*}
\end{lemma}
\begin{proof}
	The termination condition at Step 2 ensures that $\mu_k\leq \wgaugbound$, which implies that,
	\begin{equation*}
		\dfrac{\alpha_{n+1}(\mepsk{g}{k})\lambda_k}{1-\mepsk{g}{k}} \leq \kappa_\mu.
	\end{equation*}
	It follows that 
	\begin{equation*}
		\lambda_k \leq \dfrac{\kappa_\mu}{\alpha_{n+1}(\mepsk{g}{k})}(1-\mepsk{g}{k}) \leq \dfrac{\kappa_\mu}{\alpha_{n+1}(\mepsk{g}{k})}(1-\mepsmin) = \lambda_{\max}.
	\end{equation*}
	With Lemma~\ref{lem:sigma-very-successful-inexact-sum},
	\begin{equation*}
		\sigkfp{k} \geq \dfrac{\alpha_{n+1}L(1+\lambda_{\max})^2}{1-\eta_2-\eta_0-\frac{\wgaugbound}{2}} \geq \dfrac{\alpha_{n+1}L(1+\lambda_{k})^2}{1-\eta_2-\eta_0-\frac{\wgaugbound}{2}} \implies \rho_k \geq \eta_2.
	\end{equation*}
	With~\eqref{eq:sigma_update}, and since $\gamma_1<1$, we further have,
	\begin{equation*}
		\sigkfp{k} \geq \dfrac{\alpha_{n+1}L(1+\lambda_{\max})^2}{1-\eta_2-\eta_0-\frac{\wgaugbound}{2}} \implies \sigkfp{j} \leq\sigkfp{k}.
	\end{equation*}
	As a consequence, the largest value that can be reached by $\sigkfp{k}$ is obtained for the case where $\rho_{k-1}<\eta_1$ and $\sigma_{k-1} = \dfrac{\alpha_{n+1}L(1+\lambda_{\max})^2}{1-\eta_2-\eta_0-\frac{\wgaugbound}{2}} - \epsilon'$ with $\epsilon'>0$ infinitesimally small.
	In this case, 
	\begin{equation*}
		\sigkfp{k} = \sigma_{k-1}\gamma_3 \leq \dfrac{\alpha_{n+1}L(1+\lambda_{\max})^2}{1-\eta_2-\eta_0-\frac{\wgaugbound}{2}}\gamma_3 = \sigma_{\max}
	\end{equation*}
\end{proof}

%\todo{Need for an assumption for the existence of a non zero FP number lower than $\sigma_{\max}$ with $\pi_{\max}$. Either it means that $\sigkfp{k} \leq \sigma_{\max} \implies \sigkfp{k} = 0$ and the algo breaks. See where to put this assumption.}

\begin{lemma}
	\label{lem:k_bounded}
	Let $S_k = \left\{0\leq j\leq k\,|\,\rho_j\geq \eta_1\right\}$ be the set of successful iterations up to iteration $k$.
	For all $k>0$, 
	\begin{equation*}
		k \leq |S_k|\left(1+\dfrac{|\log(\gamma_1)|}{\log(\gamma_2)}\right) +\dfrac{1}{\log(\gamma_2)}\log\left(\dfrac{\sigma_{\max}}{\sigkfp{0}}\right).
	\end{equation*}
\end{lemma}

\begin{proof}
	Let $U_k = \left\{0\leq j\leq k\,|\,\rho_j< \eta_1\right\}$ be the set of unsuccessful iterations.
	From the update formula of $\sigkfp{k}$~\eqref{eq:sigma_update}, one has for each $k>0$,
	\begin{equation*}
		\forall j \in S_k,\, \gamma_1\sigkfp{j}\leq \max\left[\gamma_1\sigkfp{j},\sigma_{\min} \right]\leq \sigkfp{j+1}\text{ and } \forall i\in U_k,\, \gamma_2\sigkfp{i} \leq \sigkfp{i+1}.
	\end{equation*}
	It follows that,
	\begin{equation*}
		\sigkfp{0}\gamma_1^{|S_k|}\gamma_2^{|U_k|}\leq \sigkfp{k}.
	\end{equation*}
	With Lemma~\ref{lem:sigma_max} one obtains,
	\begin{equation*}
		|S_k|\log{\gamma_1}+|U_k|\log{\gamma_2}\leq \log\left(\dfrac{\sigma_{\max}}{\sigkfp{0}}\right).
	\end{equation*}
	Since $\gamma_2>1$, it follows that,
	\begin{equation*}
		|U_k| \leq -|S_k|\dfrac{\log(\gamma_1)}{\log(\gamma_2)} +\dfrac{1}{\log(\gamma_2)}\log\left(\dfrac{\sigma_{\max}}{\sigkfp{0}}\right).
	\end{equation*}
	Since $k = |S_k|+|U_k|$, the statement of Lemma~\ref{lem:k_bounded} holds true.
\end{proof}

\begin{theorem}
	\label{th:complexity}
	If the stopping conditions at Steps 2 and 3 are not met, Algorithm~\ref{alg:qrfp} needs at most
	\begin{equation}
		\label{eq:th_1}
		\epsilon^{-2}\kappa_s(f(x_0)-f_{low}),
	\end{equation}
	successful iterations, with 
	\begin{equation*}
		\kappa_s=\left(\dfrac{1+\beta_{n+2}}{1-\beta_{n+2}}\dfrac{1+\wgaugbound}{1-\mepsmax}\right)^2\dfrac{\sigma_{\max}}{(\eta_1-2\eta_0)(1-\gamma_{n+1})},
	\end{equation*} and at most  
	\begin{equation}
		\label{eq:th_2}
		\epsilon^{-2}\kappa_s(f(x_0)-f_{low}) \left(1+\dfrac{|\log(\gamma_1)|}{\log(\gamma_2)}\right) +\dfrac{1}{\log(\gamma_2)}\log\left(\dfrac{\sigma_{\max}}{\sigkfp{0}}\right)
	\end{equation}
	iterations to provide an iterate $\xkfp{k}$ such that $\|\nabla f(\xkfp{k})\|\leq \epsilon$.
\end{theorem}

\begin{proof}
	For all $j\in S_k$, 
	\begin{equation}
		\label{eq:hatf_diff}
		\begin{array}{ll}
			& \rho_j = \dfrac{\ffp(\xkfp{j}) - \ffp(\ckfp{j})}{\tsdifffp_k}= \dfrac{\ffp(\xkfp{j}) - \ffp(\xkfp{j+1})}{\tsdifffp_j}\geq \eta_1\\
			\implies & \ffp(\xkfp{j})-\ffp(\xkfp{j+1})\geq \eta_1\tsdifffp_j.
		\end{array}
	\end{equation}
	The condition at Step 3 ensures that $\ef(\xkfp{j})\leq \eta_0\tsdifffp_j $ and $\ef(\xkfp{j+1})\leq \eta_0\tsdifffp_j$.
	As a consequence, 
	\begin{equation}
		\label{eq:actual_diff}
		f(\xkfp{j})-f(\xkfp{j+1}) \geq \ffp(\xkfp{j})-\ffp(\xkfp{j+1})-2\eta_0\tsdifffp_j.
	\end{equation}
	Putting~\eqref{eq:hatf_diff} and~\eqref{eq:actual_diff} together, 
	\begin{equation*}
		f(\xkfp{j})-f(\xkfp{j+1}) \geq (\eta_1-2\eta_0)\tsdifffp_j.
	\end{equation*}
	As a consequence, with Lemma~\ref{lem:pseudo_reduction},
	\begin{equation}
		\label{eq:f0_flow_diff}
		\begin{array}{ll}
			f(x_0) - f_{low} & \geq (\eta_1-2\eta_0)\displaystyle\sum_{j\in S_k}\tsdifffp_j\\
			& \geq (\eta_1-2\eta_0)(1-\gamma_{n+1})\displaystyle\sum_{j\in S_k}\pseudoT(x_j,0)-\pseudoT(x_j,\skfp{j})\\
			& \geq (\eta_1-2\eta_0)(1-\gamma_{n+1})\displaystyle\sum_{j\in S_k}\|\gd_j\|^2/\sigkfp{j}\\
		\end{array}
	\end{equation} 
	From the termination condition at Step 1, before termination $\|\gkfp{j}\|$ can be bounded below as,
	\begin{equation}
		\label{eq:gk_lb}
		\widehat{\|\gkfp{j}\|} > \dfrac{1}{1+\beta_{n+2}}\dfrac{\epsilon}{1+\eg(\xkfp{j})} \geq \dfrac{1}{1+\beta_{n+2}}\dfrac{\epsilon}{1+\wgaugbound}.
	\end{equation}
	With Lemma~\ref{lem:g_actual_diff} and \ref{lem:norm_inexact},~\eqref{eq:gk_lb} yields,
	\begin{equation}
		\label{eq:gd_lb}
		\begin{array}{ll}
			\|\gd_j\| & \geq \|\gkfp{j}\|(1-\mepsk{g}{j})\\
			& \geq \widehat{\|\gkfp{j}\|}(1-\beta_{n+2})(1-\mepsk{g}{j})\\
			& > (1-\beta_{n+2})(1-\mepsmax)\dfrac{1}{1+\beta_{n+2}}\dfrac{\epsilon}{1+\wgaugbound}.
		\end{array}
	\end{equation}
	With Lemma~\ref{lem:sigma_max}, putting~\eqref{eq:f0_flow_diff} and~\eqref{eq:gd_lb} together yields,
	\begin{equation}
		\label{eq:f0_flow_diff_2}
		\begin{array}{ll}
			f(x_0) - f_{low} & \geq (\eta_1-2\eta_0)(1-\gamma_{n+1})\displaystyle\sum_{j\in S_k}\|\gd_j\|^2/\sigkfp{j}\\
			& \geq (\eta_1-2\eta_0)(1-\gamma_{n+1})\displaystyle\sum_{j\in S_k}\|\gd_j\|^2/\sigma_{\max}\\
			& \geq (\eta_1-2\eta_0)(1-\gamma_{n+1})\displaystyle\sum_{j\in S_k}\left(\dfrac{1-\beta_{n+2}}{1+\beta_{n+2}}\dfrac{1-\mepsmax}{1+\wgaugbound}\right)^2\dfrac{\epsilon^2}{\sigma_{\max}}\\
			& \geq (\eta_1-2\eta_0)(1-\gamma_{n+1})|S_k|\left(\dfrac{1-\beta_{n+2}}{1+\beta_{n+2}}\dfrac{1-\mepsmax}{1+\wgaugbound}\right)^2\dfrac{\epsilon^2}{\sigma_{\max}}\\
			& = \dfrac{1}{\kappa_s}|S_k|\epsilon^2.
		\end{array}
	\end{equation}
	As a consequence, 
	\begin{equation*}
		|S_k| \leq (f(x_0)-f_{low})\kappa_s\epsilon^{-2},
	\end{equation*}
	which proves~\eqref{eq:th_1}.
	With Lemma~\ref{lem:k_bounded},~\eqref{eq:th_1} yields,
	\begin{equation*}
		\begin{array}{ll}
			k & \leq |S_k|\left(1+\dfrac{|\log(\gamma_1)|}{\log(\gamma_2)}\right) +\dfrac{1}{\log(\gamma_2)}\log\left(\dfrac{\sigma_{\max}}{\sigkfp{0}}\right) \\
			& \leq \epsilon^{-2}(f(x_0)-f_{low})\kappa_s\left(1+\dfrac{|\log(\gamma_1)|}{\log(\gamma_2)}\right) +\dfrac{1}{\log(\gamma_2)}\log\left(\dfrac{\sigma_{\max}}{\sigkfp{0}}\right), \\
		\end{array}
	\end{equation*}
	which proves~\eqref{eq:th_2}.
\end{proof}

\bibliographystyle{plain}
\bibliography{ref}

\end{document}

%% file: notations.tex
\newcommand{\pmprt}{\texttt{r-}MPR2 }
\newcommand{\minu}[1]{\min_{#1}}
\newcommand{\argminu}[1]{\text{argmin}{#1}}
\usepackage{dsfont}
\newcommand{\F}{\mathds{F}}
\newcommand{\R}{\mathds{R}}
\newcommand{\ffp}{\widehat{f}} % finite precision computation of f
\newcommand{\gfp}{\widehat{g}} % finite precision computation of f
\newcommand{\gd}{\widetilde{g}} % exact descent direction
\newcommand{\xkfp}[1]{\widehat{x}_{#1}} % computed incumbent
\newcommand{\skfp}[1]{\widehat{s}_{#1}} % computed step
\newcommand{\ckfp}[1]{\widehat{c}_{#1}} % computed candidate
\newcommand{\skact}[1]{\widetilde{s}_{#1}^a} % actual step
\newcommand{\gkact}[1]{\widetilde{g}_{#1}^a} % actual descent direction
\newcommand{\sigkfp}[1]{\widehat{\sigma}_{#1}} % computed candidate
\newcommand{\deltak}[2]{\delta_{#2}^{#1}} % particular delta related to quantity #1 and iteration #2
\newcommand{\fkfp}[1]{\widehat{f}_{#1}} % computed objective function at xk
\newcommand{\fkcfp}[1]{\widehat{f}_{#1}^+} % computed objective function at xk
\newcommand{\gkfp}[1]{\widehat{g}_{#1}} % computed gradient
\newcommand{\ef}{\omega_f} %  |f(x)-\ffp(x)| < \ef
\newcommand{\eg}{\omega_g} % ||df(x)-g(x)| < \eg||g(x)||
\newcommand{\approxT}{\overline{T}} % order 1 taylor serie with g
\newcommand{\fl}{\mathop{\textup{fl}}} % finite precision computation
\newcommand{\pseudoT}{\widetilde{T}} % pseudo taylor serie
\newcommand{\tsdifffp}{\widehat{\Delta T}} % finite precision computation of approximated taylor diff 
\newcommand{\wgaugbound}{\kappa_\mu} % bound on \mu 
\newcommand{\xpreck}[1]{\pi_{#1}^x} % precision level of xk
\newcommand{\fpreck}[1]{\pi_{#1}^{f^+}} % precision level for evaluation of f(ck)
\newcommand{\cpreck}[1]{\pi_{#1}^c} % precision level of ck
\newcommand{\fprecrecompk}[1]{\pi_{#1}^{f}} % precision level for re-evaluation of f(xk)
\newcommand{\gpreck}[1]{\pi_{#1}^g} % precision level for gradient evaluation
\newcommand{\dmixk}[1]{{\delta'_{#1}}} % rounding error of ck computation, mixes error due to addition and casting error to precision level pi_c
\newcommand{\umixk}[1]{{u'_{#1}}} % upper bound on \dmixk

\newcommand{\meps}{u} % machine epsilon
\newcommand{\mepsk}[2]{\meps^{#1}_{#2}} % machine epsilon related to #1 (c,s,x,...) at iteration #2
\newcommand{\mepsmax}{\meps_{max}} % greatest machine epsilon, related to lowest machine precision available
\newcommand{\mepsmin}{\meps_{min}} % lowest machine epsilon, related to highest machine precision available

%% file: figures/actual_candidate.tex
\newcommand{\cross}[3]{\draw (#1-#3,#2-#3)--(#1+#3,#2+#3);
\draw (#1-#3,#2+#3)--(#1+#3,#2-#3);
}
\usetikzlibrary{decorations.pathreplacing}
\begin{tikzpicture}
	\def\csize{0.1}
	\def\xkx{0}
	\def\xky{0}
	\def\gkx{-1.6}
	\def\gky{-1.2}
	\def\gkpx{-1.2}
	\def\gkpy{-1.6}
	
	\def\skx{1.5}
	\def\sky{1.125}
	\def\skpx{1.125}
	\def\skpy{1.5}
	\def\ckpx{1}
	\def\ckpy{2.6}
	
	\cross{\xkx}{\xky}{\csize};
	\node at (\xkx,\xky) [right]{$\xkfp{k}$};
	\draw[dashed,->] (\xkx,\xky) -- (\gkx,\gky) node [pos=1.1] {$\gkfp{k}$};
	\draw[dashed] (\xkx,\xky) -- (-\gkx,-\gky);
	\draw[->] (\xkx,\xky) -- (\gkpx,\gkpy) node [pos=1.1]{$\gd_k$};
	\draw (\xkx,\xky) -- (-\gkpx,-\gkpy);
	
	\cross{\skx}{\sky}{\csize};
	\node at (\skx,\sky) [right] {$\xkfp{k}-\dfrac{\gkfp{k}}{\sigma_k} = c_k$};
	
	\cross{\skpx}{\skpy}{\csize};
	\node at (\skpx,\skpy) [left] {$\xkfp{k}+\skfp{k}$};
	
	\cross{\ckpx}{\ckpy}{\csize};
	\node (hck) at (\ckpx,\ckpy) [above] {$\ckfp{k}$};
	\draw[decorate,decoration={brace,amplitude=7pt,raise=3pt}] (\ckpx,\ckpy) -- (\skpx,\skpy) node [midway, xshift=40pt] {$\xkfp{k}\deltak{+}{k} + \skfp{k}\deltak{+}{k}$};
	
	\draw[dotted,->] (hck) -- (-0.7*\ckpx,-0.7*\ckpy)node [pos=1.1] {$\gkact{k}$};

\end{tikzpicture}